\documentclass[reqno]{amsart}
\usepackage{amsmath}
\usepackage{amscd,amsthm,amsfonts,amsopn,amssymb}
\usepackage[linktocpage]{hyperref}
\usepackage{epsfig}
\usepackage{graphicx}
\numberwithin{equation}{section}

\makeatletter
\def\subsection{\@startsection{subsection}{2}%
  \z@{0.0\linespacing}{-.5em}%
  {\normalfont\bfseries}}
\makeatother

\newtheorem{theorem}{Theorem}[section]

\newtheorem{proposition}[theorem]{Proposition}

\newtheorem{definition}[theorem]{Definition}
\newtheorem{lemma}[theorem]{Lemma}

\theoremstyle{remark}
\newtheorem{remark}[theorem]{Remark}
\DeclareMathOperator{\supp}{supp\,}

\long\def\symbolfootnote[#1]#2{\begingroup%
\def\thefootnote{\fnsymbol{footnote}}\footnote[#1]{#2}\endgroup} 

\allowdisplaybreaks
\begin{document}
\title[The defocusing energy-supercritical cubic wave equation]{Global Well-Posedness and Scattering for the Defocusing Energy-Supercritical Cubic Nonlinear Wave Equation}
\author[Aynur Bulut]{Aynur Bulut}\address{Department of Mathematics, University of Texas at Austin}\email{abulut@math.utexas.edu}
\thanks{Date: \today}
\maketitle
\begin{abstract}
In this paper, we consider the defocusing cubic nonlinear wave equation $u_{tt}-\Delta u+|u|^2u=0$ in the energy-supercritical regime, in dimensions $d\geq 6$, with no radial assumption on the initial data. We prove that if a solution satisfies an a priori bound in the critical homogeneous Sobolev space throughout its maximal interval of existence, that is, $u\in L_t^\infty(\dot{H}_x^{s_c}\times\dot{H}_x^{s_c-1})$, then the solution is global and it scatters.  Our analysis is based on the methods of the recent works of Kenig-Merle \cite{KenigMerleSupercritical} and Killip-Visan \cite{KillipVisanSupercriticalNLS,KillipVisanSupercriticalNLW3D} treating the energy-supercritical nonlinear Schr\"odinger and wave equations.
\end{abstract} 
\tableofcontents
\parskip=8pt
\section{Introduction}
We consider the initial value problem for the defocusing nonlinear wave equation with cubic nonlinearity $F(u)=|u|^2u$ in the energy-supercritical regime, in dimensions $d\geq 6$.  More precisely, we study
\begin{align*}
\textrm{(NLW)}\quad\left\lbrace \begin{array}{rl}u_{tt}-\Delta u+|u|^2u&=0\\
(u,u_t)|_{t=0}&=(u_0,u_1)\in \dot{H}_x^{s_c}\times \dot{H}_x^{s_c-1}(\mathbb{R}^d),
\end{array}
\right.
\end{align*}
where $u(t,x)$ is a real-valued function on $I\times\mathbb{R}^d$ with $d\geq 6$ and $0\in I\subset\mathbb{R}$ is a time interval.

Before explaining the terminology ``energy-supercritical'' let us first recall the notion of criticality.  There is a natural scaling associated to the initial value problem (NLW).  More precisely, if we set
\begin{align*}
u_\lambda(t,x)=\lambda u(\lambda t,\lambda x) \quad \lambda>0,
\end{align*}
then the map $u\mapsto u_\lambda$ maps a solution of (NLW) to another solution of (NLW) and 
\begin{align}
\lVert (u_\lambda,u_{\lambda,t})|_{t=0}\rVert_{\dot{H}_x^{s_c}\times\dot{H}_x^{s_c-1}}=\lVert (u_0,u_1)\rVert_{\dot{H}_x^{s_c}\times\dot{H}_x^{s_c-1}},\label{l1}
\end{align}
where we define the {\it critical regularity} as $s_c=\frac{d-2}{2}$.  In the case $s_c=1$, the above scaling leaves the {\it energy},
\begin{align*}
E(u(t),u_t(t))=\int_{\mathbb{R}^d}\frac{1}{2}|\nabla u(t)|^2+\frac{1}{2}|u_t(t)|^2+\frac{1}{4}|u(t)|^4dx
\end{align*}
invariant.  We note that, in view of the cubic nonlinearity, dimension $d>4$ corresponds to the range $s_c>1$, and is therefore known as the {\it energy-supercritical} regime for (NLW).

In the present work, we study (NLW) with initial data lying in the critical homogeneous Sobolev space $\dot{H}_x^{s_c}\times\dot{H}_x^{s_c-1}$ in the energy-supercritical regime $s_c>1$, in dimensions $d\geq 6$, with no radial assumption on the initial data.  

We consider {\it solutions} to (NLW), that is, functions $u:I\times\mathbb{R}^d\rightarrow\mathbb{R}$ such that for every $K\subset I$ compact, $(u,u_t)\in C_t(K;\dot{H}_x^{s_c}\times\dot{H}_x^{s_c-1})$, $u\in L_{t,x}^{d+1}(K\times\mathbb{R}^d)$, and satisfying the {\it Duhamel formula}
\begin{align*}
u(t)&=\mathcal{W}(t)(u_0,u_1)+\int_0^t \frac{\sin((t-t')|\nabla|)}{|\nabla|}F(u(t'))dt'
\end{align*}
for every $t\in I$, where $0\in I\subset\mathbb{R}$ is a time interval and the wave propagator
\begin{align*}
\mathcal{W}(t)(u_0,u_1)=\cos(t|\nabla|)u_0+\frac{\sin(t|\nabla|)}{|\nabla|}u_1
\end{align*}
is the solution to the linear wave equation with initial data $(u_0,u_1)$.  

We refer to $I$ as the {\it interval of existence} of $u$, and we say that $I$ is the {\it maximal interval of existence} if $u$ cannot be extended to any larger time interval.  We say that $u$ is a {\it global solution} if $I=\mathbb{R}$, and that $u$ is a {\it blow-up solution} if $\lVert u\rVert_{L_{t,x}^{d+1}(I\times\mathbb{R}^d)}=\infty$.

In this paper, we prove that if $u$ is a solution to (NLW) which is uniformly bounded in the critical space for all times in its maximal interval of existence, then it is defined globally in time and scatters.  

More precisely, our main result is the following:
\begin{theorem}
\label{l2}
Let $d\geq 6$ and $s_c=\frac{d-2}{2}$.  Assume $u:I\times\mathbb{R}^d\rightarrow\mathbb{R}$ is a solution to (NLW) with maximal interval of existence $I\subset\mathbb{R}$ satisfying
\begin{align}
(u,u_t)\in L_t^\infty(I;\dot{H}_x^{s_c}\times \dot{H}_x^{s_c-1}).
\label{l3}
\end{align} 

Then $u$ is global and 
\begin{align*}
\lVert u\rVert_{L_{t,x}^{d+1}(\mathbb{R}\times\mathbb{R}^d)}\leq C.
\end{align*}
for some constant $C=C(\lVert (u,u_t)\rVert_{L_t^\infty(I;\dot{H}_x^{s_c}\times\dot{H}_x^{s_c-1})})$.

Moreover, $u$ scatters in the sense that there exist unique $(u_0^\pm,u_1^\pm)\in \dot{H}_x^{s_c}\times \dot{H}_x^{s_c-1}$ such that
\begin{align*}
\lim_{t\rightarrow \pm\infty} \lVert (u(t),u_t(t))-(\mathcal{W}(t)(u_0^\pm,u_1^\pm),\partial_t \mathcal{W}(t)(u_0^\pm,u_1^\pm))\rVert_{\dot{H}_x^{s_c}\times\dot{H}_x^{s_c-1}}=0.
\end{align*}
\end{theorem}

We note that when the cubic nonlinearity $F(u)=|u|^2u$ is replaced by the $d$-dimensional energy-supercritical nonlinearity $|u|^pu$, $p>\frac{4}{d-2}$, the above theorem was proved by Kenig and Merle \cite{KenigMerleSupercritical} in $d=3$ for radial initial data, and by Killip and Visan \cite{KillipVisanSupercriticalNLW3D} for general data in $d=3$ with even values of $p$ and also in \cite{KillipVisanSupercriticalNLWradial} for $d\geq 3$ and radial initial data with a specified range of $p$.

The contribution of the present work to the study of the energy-supercritical regime is to consider the case of higher dimensions $d\geq 6$ with no radial assumption on the initial data.  The restriction to the cubic nonlinearity in our considerations mainly serves to simplify the estimates required for the local theory.

The main tool which allows us to consider non-radial initial data, as in the Schr\"odinger context \cite{KillipVisanSupercriticalNLS}, is to prove that certain solutions to (NLW) have finite energy.  This result makes use of the double Duhamel technique \cite{CKSTT,Tao} which is used for the same purpose in \cite{KillipVisanECritical,KillipVisanSupercriticalNLS}.  In the present context, the restriction to dimensions $d\geq 6$ appears as a consequence of our use of this technique; see the discussion in Section $3$ for a more detailed account.

We also remark that similar results showing that the boundedness of a critical norm implies global well-posedness are known for Navier-Stokes, which is also a supercritical problem with respect to the control given by the known conservation laws and monotonicity formulae; see the work of Escauriaza, Seregin, and \v Sver\'ak \cite{EscauriazaSereginSverak} as well as Kenig and Koch \cite{KenigKoch}.

In the case $s_c=1$ with the energy critical defocusing nonlinearity $|u|^{4/(d-2)}u$, local well-posedness for the initial value problem (NLW) has been studied in a number of papers; see, for instance,\cite{BulutCzubakLiPavlovicZhang,GinibreSoferVelo,KenigMerleNLW,LindbladSogge,Pecher,ShatahStruwe2,ShatahStruwe,Sogge}.  Global well-posedness in the defocusing case was obtained in a series of works \cite{BahouriGerard,Grillakis1,Grillakis2,Kapitanski,Nakanishi,Pecher,Rauch,ShatahStruwe1,ShatahStruwe,Struwe,Tao1}.  In particular, Struwe \cite{Struwe} obtained the global well-posedness for energy critical (NLW) with radial initial data in $d=3$, while Grillakis \cite{Grillakis1} removed the radial assumption in this dimension.  The global well-posedness and persistence of regularity was shown for $3\leq d\leq 5$ by Grillakis \cite{Grillakis2}, and for $d\geq 3$ by Shatah and Struwe \cite{ShatahStruwe1,ShatahStruwe2,ShatahStruwe} and Kapitanski \cite{Kapitanski}.

We remark that in all of the works cited in the previous paragraph, the key property in obtaining global well-posedness results for the energy critical (NLW) is an immediate uniform control in time of the critical norm $\dot{H}_x^1\times L_x^2$ by virtue of the conservation of energy.  It is also important to note that monotonicity formulae like the Morawetz identity have the critical scaling in all of these results.

In the case $s_c>1$, the energy supercritical regime, the global behavior of solutions to (NLW) is a more delicate matter, as in this context we do not have instantaneous access to any conservation law at the critical regularity.  In view of the energy critical theory, it is then natural to impose an a priori uniform in time control of the critical norm $\dot{H}_x^{s_c}\times\dot{H}_x^{s_c-1}$ to compensate for the lack of such a conservation law.  This is the reason why we have the assumption $(u,u_t)\in L_t^\infty(I;\dot{H}_x^{s_c}\times\dot{H}_x^{s_c-1})$ in Theorem $\ref{l2}$.  However, the difficulty that the scaling of the a priori bound ($\ref{l3}$) no longer matches the scaling of the monotonicity formulae, namely the Morawetz identity, remains to be overcome.  Thus, one must proceed in a different manner than in the energy critical case.

A similar difficulty, where the monotonicity formula has a different scaling than the known conservation laws, also appears in the study of the nonlinear Schr\"odinger equation, and the techniques developed in that setting will play an important role in our analysis.  Accordingly, we now briefly describe the approach that we follow in this paper.  For a detailed discussion, we refer the reader to Section $3$.  To prove Theorem $\ref{l2}$ we argue by contradiction: assuming that the theorem fails, one constructs a minimal blowup solution using the concentration compactness/rigidity approach introduced by Kenig-Merle in their work \cite{KenigMerleNLS,KenigMerleNLW,KenigMerleSupercritical}.  Then, using a further reduction obtained by Killip-Tao-Visan \cite{KillipTaoVisanCubic} and Killip-Visan \cite{KillipVisanECritical,KillipVisanSupercriticalNLS,KillipVisanSupercriticalNLW3D}, we conclude that there exists a special solution satisfying one of three possible scenarios: the finite time blow-up solution, the soliton-like solution, and the low-to-high frequency cascade solution.  To conclude the argument, we then show that each such scenario cannot exist.
\vspace{0.05in}

\subsection*{Organization of the paper}

We now outline the remainder of this paper.  In Section $2$, we introduce our notation and present some preliminaries for our discussion.  In Section $3$, we give a detailed overview of the proof of Theorem $\ref{l2}$.  Section $4$ is devoted to the study of the local theory (local well-posedness and stability), while in Section $5$ we state and prove a lemma as a consequence of the finite speed of propagation that will be used in Sections $6$ and $8$.
In Section $6$, we rule out the finite-time blow-up scenario.  In Section $7$, we prove an additional decay result for the soliton-like and low-to-high frequency cascade scenarios.  This result is then used to rule out these two cases in Sections $8$ and $9$ respectively.
We conclude the paper with a brief Appendix, in which we provide the details of some arguments used in the main body of the paper.

\section{Preliminaries}

In this section, we introduce the notation and some basic estimates that we use throughout the paper.  
For any time interval $I\subseteq \mathbb{R}$, we write $L_t^qL_x^r(I\times\mathbb{R}^d)$ to denote the spacetime norm
\begin{align*}
\lVert u\rVert_{L_t^qL_x^r}&=\left(\int_{\mathbb{R}}\left(\int_{\mathbb{R}^d} |u(t,x)|^rdx\right)^\frac{q}{r}dt\right)^\frac{1}{q}
\end{align*}
with the standard definitions when $q$ or $r$ is equal to infinity.  In the case $q=r$, we shorten the notation $L_t^qL_x^r$ and write $L_{t,x}^q$.  

We write $X\lesssim Y$ to indicate that there exists a constant $C>0$ such that $X\leq CY$.  We use the symbol $\nabla$ for the derivative operator in the space variable.  

\noindent In what follows, we define the Fourier transform on $\mathbb{R}^d$ by 
\begin{align*}
\hat{f}(\xi)&:=(2\pi)^{-d/2}\int_{\mathbb{R}^d} e^{-ix\cdot \xi}f(x)dx.
\end{align*}

\noindent We also define the homogeneous Sobolev space $\dot{H}_x^s(\mathbb{R}^d)$, $s\in\mathbb{R}$ via the norm
\begin{align*}
\lVert f\rVert_{\dot{H}_x^{s}}&:=\lVert |\nabla|^sf\rVert_{L_x^2}
\end{align*}
where the fractional differentiation operator is given by
\begin{align*}
\widehat{|\nabla|^sf}(\xi)&:=|\xi|^s\hat{f}(\xi).
\end{align*}

We use $\mathcal{W}(t)$ to denote the linear wave propagator associated to (NLW).  In physical space the operator is given by
\begin{align*}
\mathcal{W}(t)(f,g)&=\cos(t|\nabla|)f+\frac{\sin(t|\nabla|)}{|\nabla|}g
\end{align*}
or, equivalently, in frequency space it is written as
\begin{align*}
\widehat{\mathcal{W}(t)}(f,g)(\xi)=\cos(t|\xi|)\hat{f}(\xi)+\frac{\sin(t|\xi|)}{|\xi|}\hat{g}(\xi).
\end{align*}
In particular, in terms of the explicit form of the propagator, we recall the following standard dispersive estimate.
\begin{proposition}[Dispersive estimate, \cite{ShatahStruwe}]
\label{l4}
For any $d\geq 2$, $2\leq p<\infty$ and $t\neq 0$ we have
\begin{align}
\bigg\lVert \frac{e^{it|\nabla|}}{|\nabla|} f\bigg\rVert_{L_x^p}&\lesssim |t|^{-\frac{d-1}{2}\left(1-\frac{2}{p}\right)}\lVert |\nabla|^{\frac{d-1}{2}-\frac{d+1}{p}}f\rVert_{L_x^{p'}}.
\end{align}
In particular,
\begin{align}
\bigg\lVert \frac{\sin(t|\nabla|)}{|\nabla|}f\bigg\rVert_{L_x^p(\mathbb{R}^d)}&\lesssim |t|^{-\frac{(d-1)}{2}\left(1-\frac{2}{p}\right)}\lVert |\nabla|^{\frac{d-1}{2}-\frac{d+1}{p}}f\rVert_{L_x^{p'}(\mathbb{R}^d)}
\label{l5}
\end{align}
and
\begin{align*}
\bigg\lVert \frac{\cos(t|\nabla|)}{|\nabla|^2}g\bigg\rVert_{L_x^p(\mathbb{R}^d)}&\lesssim |t|^{-\frac{(d-1)}{2}\left(1-\frac{2}{p}\right)}\lVert |\nabla|^{\frac{d-3}{2}-\frac{d+1}{p}}g\rVert_{L_x^{p'}(\mathbb{R}^d)},
\end{align*}
for all $f,g\in \mathcal{S}(\mathbb{R}^d)$, where $\frac{1}{p'}+\frac{1}{p}=1$.
\end{proposition}

For $s\geq 0$, we say that a pair of exponents $(q,r)$ is $\dot{H}_x^s$-{\it wave admissible} if $q,r\geq 2$, $r<\infty$ and it satisfies 
\begin{align*}
\frac{1}{q}+\frac{d-1}{2r}\leq \frac{d-1}{4},\\
\frac{1}{q}+\frac{d}{r}=\frac{d}{2}-s.
\end{align*}

The {\it Strichartz estimates} then read as follows; for a proof, see \cite{GinibreVelo,KeelTao,Sogge}.  
Assume $u:I\times\mathbb{R}^d\rightarrow\mathbb{R}$ with time interval $0\in I\subset\mathbb{R}$ is a solution to the nonlinear wave equation
\begin{align*}
\left\lbrace \begin{array}{rl}u_{tt}-\Delta u+F&=0\\
(u,u_t)|_{t=0}&=(u_0,u_1)\in \dot{H}_x^{\mu}\times \dot{H}_x^{\mu-1}(\mathbb{R}^d), \quad \mu\in\mathbb{R}.
\end{array}
\right.
\end{align*}
Then 
\begin{align}
 \label{l6}&\lVert |\nabla|^{s}u\rVert_{L_t^qL_x^r}+\lVert |\nabla|^{s-1}u_t\rVert_{L_t^qL_x^r}+\lVert |\nabla|^{\mu}u\rVert_{L_t^\infty L_x^2}+\lVert |\nabla|^{\mu-1}u_t\rVert_{L_t^\infty L_x^2 }\\
\nonumber&\hspace{1.5in}\lesssim \lVert (u_0,u_1)\rVert_{\dot{H}_x^{\mu}\times \dot{H}_x^{\mu-1}}+\lVert |\nabla|^{\tilde{s}}F\rVert_{L_t^{\tilde{q}'}L_x^{\tilde{r}'}}
\end{align}
for $s\geq 0$, where the pair $(q,r)$ is $\dot{H}_x^{\mu-s}$-wave admissible and the pair $(\tilde{q},\tilde{r})$ is $\dot{H}_x^{1+\tilde{s}-\mu}$-wave admissible.

We also define the following {\it Strichartz norms}.  For each $I\subset\mathbb{R}$ and $s\geq 0$, we set
\begin{align*}
\lVert u\rVert_{S_s(I)}&=\sup_{(q,r)\,\dot{H}_x^s-\textrm{wave admissible}} \lVert u\rVert_{L_t^qL_x^r(I\times\mathbb{R}^d)},\\
\lVert u\rVert_{N_s(I)}&=\inf_{(q,r)\,\dot{H}_x^s-\textrm{wave admissible}} \lVert u\rVert_{L_t^{q'}L_x^{r'}(I\times\mathbb{R}^d)}.
\end{align*}

\noindent Taking the supremum over $(q,r)$ $\dot{H}_x^{\mu-s}$-wave admissible and the infimum over $(\tilde{q},\tilde{r})$ $\dot{H}_x^{1+\tilde{s}-\mu}$-wave admissible pairs in ($\ref{l6}$), we also have,
\begin{align*}
\lVert |\nabla|^s u\rVert_{S_{\mu-s}(I)}+\lVert |\nabla|^{s-1}u_t\rVert_{S_{\mu-s}(I)}\lesssim \lVert (u_0,u_1)\rVert_{\dot{H}_x^{\mu}\times\dot{H}_x^{\mu-1}}+\lVert |\nabla|^{\tilde{s}}F\rVert_{N_{1+\tilde{s}-\mu}(I)}.
\end{align*}

\vspace{0.05in}
We next recall some basic facts from Littlewood-Paley theory that will be used frequently in the sequel.  Let $\phi(\xi)$ be a real valued radially symmetric bump function supported in the ball $\{\xi\in\mathbb{R}^d:|\xi|\leq 2\}$ which equals $1$ on the ball $\{\xi\in\mathbb{R}^d:|\xi|\leq 1\}$.  For any dyadic number $N=2^k$, $k\in\mathbb{Z}$, we define the following Littlewood-Paley operators:
\begin{align*}
\widehat{P_{\leq N}f}(\xi)&=\phi(\xi/N)\hat{f}(\xi),\\
\widehat{P_{>N}f}(\xi)&=(1-\phi(\xi/N)\hat{f}(\xi),\\
\widehat{P_Nf}(\xi)&=(\phi(\xi/N)-\phi(2\xi/N))\hat{f}(\xi).
\end{align*}

\noindent Similarly, we define $P_{<N}$ and $P_{\geq N}$ with 
\begin{align*}
P_{<N}=P_{\leq N}-P_N,\quad P_{\geq N}=P_{>N}+P_N,
\end{align*}
and also
\begin{align*}
P_{M<\cdot\leq N}:=P_{\leq N}-P_{\leq M}=\sum_{M<N_1\leq N} P_{N_1}
\end{align*}
whenever $M\leq N$.

These operators commute with one another, with derivative operators and with the wave propagator $\mathcal{W}(t)(f,g)$.  Moreover, they are bounded on $L_x^p$ for $1\leq p\leq \infty$ and obey the following {\it Bernstein inequalities}, 
\begin{align*}
\lVert P_Nf\rVert_{L_x^q}&\lesssim N^{\frac{d}{p}-\frac{d}{q}}\lVert P_Nf\rVert_{L_x^p},\\
\lVert |\nabla|^{\pm s}P_Nf\rVert_{L_x^p}&\sim N^{\pm s} \lVert P_Nf\rVert_{L_x^p},\\
\lVert P_{\leq N}f\rVert_{L_x^q}&\lesssim N^{\frac{d}{p}-\frac{d}{q}}\lVert P_{\leq N}f\rVert_{L_x^p}.
\end{align*}
with $s\geq 0$ and $1\leq p\leq q\leq\infty$.

\vspace{0.1in}

We also recall the following Morawetz estimate for the wave equation.
\begin{theorem} [Morawetz estimate \cite{Morawetz,Morawetz2}]
\label{l7}
Assume $u:I\times\mathbb{R}^d\rightarrow\mathbb{R}$ is a solution to (NLW).  Then we have
\begin{align*}
\int_{I}\int_{\mathbb{R}^d} \frac{|u(t,x)|^4}{|x|}dxdt\leq CE(u,u_t).
\end{align*}
\end{theorem}

We end this section by noting some basic facts concerning the fractional derivative operator.
\begin{remark}
Suppose $\phi\in C_0^\infty(\mathbb{R}^d)$, where $C^\infty_0$ denotes the space of smooth functions having compact support.  Then for all nonnegative integers $s$ and all $p\geq 1$ we have $|\nabla|^s\phi\in L_x^p$, while for all $s>0$ and all $p\in [2,d)$, we have $|\nabla|^s\phi\in L_x^p$.  
\label{l8}
\end{remark}

We also note a (simple) version of the chain rule which allows us to compute the fractional derivative of a composition with a linear function.
\begin{remark}
\label{l9}For all $s>0$, $|\nabla|^s\left[u(\alpha \cdot)\right](x)=\alpha^s (|\nabla|^su)(\alpha x)$.
\end{remark}

\parskip=8pt
\section{Overview of the proof of Theorem $\ref{l2}$}

We now give a brief outline of the proof of our main result, Theorem \ref{l2}.  The approach we pursue here follows the methods introduced by Kenig and Merle \cite{KenigMerleNLS,KenigMerleNLW} and Killip, Tao, and Visan \cite{KillipTaoVisanCubic}, and developed in the works \cite{KenigMerleSupercritical,KenigMerleH12,KillipVisanSupercriticalNLS,KillipVisanECritical,KillipVisanSupercriticalNLW3D}.

As we have mentioned in the introduction, the proof of Theorem $\ref{l2}$ is an argument by contradiction and consists of the following components:

\subsection{Concentration compactness}
The first ingredient in establishing Theorem \ref{l2} is a concentration compactness result in the form of a profile decomposition theorem for solutions of the linear wave equation.  In a broad sense, it asserts that any bounded sequence of initial data in the critical space $\dot{H}_x^{s_c}\times \dot{H}_x^{s_c-1}$ can be decomposed up to a subsequence as the sum of a superposition of profiles and an error term.  The profiles are asymptotically orthogonal and the remainder term is small in a Strichartz norm.  The idea behind this decomposition is to compensate for the lack of compactness of the linear wave propagator $\mathcal{W}(t)$ as a map from the space $\dot{H}_x^{s_c}\times \dot{H}_x^{s_c-1}$ to the Strichartz space $S_{s_c}(\mathbb{R})$.  

In the present context, the higher dimensional version of the profile decomposition with initial data lying in $\dot{H}_x^{s_c}\times \dot{H}_x^{s_c-1}$ reads as follows:
\begin{theorem} [Profile decomposition \cite{Bulut}]
\label{l10}Let $s_c=\frac{d-2}{2}$ and $(u_{0,n},u_{1,n})_{n\in\mathbb{N}}$ be a bounded sequence in $\dot{H}_x^{s_c}\times \dot{H}_x^{s_c-1}(\mathbb{R}^d)$ with $d\geq 6$.  Then there exists a subsequence of $(u_{0,n},u_{1,n})$ (still denoted $(u_{0,n},u_{1,n})$), a sequence of profiles $(V_0^j,V_1^j)_{j\in \mathbb{N}}\subset \dot{H}_x^{s_c}\times\dot{H}_x^{s_c-1}(\mathbb{R}^d)$, and a sequence of triples $(\epsilon_n^j,x_n^j,t_n^j)\in \mathbb{R}^+\times\mathbb{R}^d\times\mathbb{R}$, which are orthogonal in the sense that for every $j\neq j'$,
\begin{align*}
\frac{\epsilon_n^j}{\epsilon_n^{j'}}+\frac{\epsilon_n^{j'}}{\epsilon_n^j}+\frac{|t_n^j-t_n^{j'}|}{\epsilon_n^j}+\frac{|x_n^j-x_n^{j'}|}{\epsilon_n^j}\mathop{\longrightarrow}_{n\rightarrow\infty}
\infty,
\end{align*}
and for every $l\geq 1$, if 
\begin{align*}
V^j=\mathcal{W}(t)(V_0^j,V_1^j)\quad  and \quad
V_n^j(t,x)=\frac{1}{(\epsilon_n^j)}V^j\left(\frac{t-t_n^j}{\epsilon_n^j},\frac{x-x_n^j}{\epsilon_n^j}\right),
\end{align*}
then
\begin{align*}
(u_{0,n}(x),u_{1,n}(x))&=\sum_{j=1}^l (V_n^j(0,x),\partial_t V_n^j(0,x))+(w_{0,n}^l(x),w_{1,n}^l(x))
\end{align*}
with
\begin{align*}
\limsup_{n\rightarrow\infty}\lVert \mathcal{W}(t)(w_{0,n}^l,w_{1,n}^l)\rVert_{L_t^qL_x^r}\mathop{\longrightarrow}_{l\rightarrow\infty}0
\end{align*}
for every $(q,r)$ an $\dot{H}_x^{s_c}$-wave admissible pair with $q,r\in (2,\infty)$. For all $l\geq 1$, we also have, 
\begin{align*}
\lVert u_{0,n}\rVert_{\dot{H}_x^{s_c}}^2+\lVert u_{1,n}\rVert_{\dot{H}_x^{s_c-1}}^2&=\sum_{j=1}^l \left[\lVert V^j_0\rVert_{\dot{H}_x^{s_c}}^2+\lVert V^j_1\rVert_{\dot{H}_x^{s_c-1}}^2\right]+\lVert w_{0,n}^l\rVert_{\dot{H}_x^{s_c}}^2+\lVert w_{1,n}^l\rVert_{\dot{H}_x^{s_c-1}}^2\\
&\hspace{2.5in}+o(1),\quad n\rightarrow\infty. 
\end{align*}
\end{theorem}

For initial data in $\dot{H}_x^1\times L_x^2$, the profile decomposition for the wave equation was established by Bahouri and Gerard \cite{BahouriGerard} in dimension $3$ and was extended to higher dimensions by the author in \cite{Bulut}.  Roughly speaking, the proof of Theorem $\ref{l10}$ is obtained by observing that for any sequence of initial data $\{(u_{0,n},u_{1,n})\}\subset \dot{H}_x^{s_c}\times \dot{H}_x^{s_c-1}$, the sequence $\{(|\nabla|^{s_c-1}u_{0,n},|\nabla|^{s_c-1}u_{1,n})\}$ lies in the energy space $\dot{H}_x^1\times L_x^2$.  Applying the energy-critical profile decomposition to this new sequence, the result then follows from an application of the Sobolev embedding.  For more details, we refer the reader to \cite{BahouriGerard,Bulut}.

\subsection{Existence of minimal blow-up solutions.}
The first part in the ``concentration compactness + rigidity'' method introduced by Kenig and Merle \cite{KenigMerleNLS,KenigMerleNLW} consists of reducing the argument to the study of minimal blow-up solutions to (NLW).  Informally speaking, this reduction is a consequence of the observation that if Theorem $\ref{l2}$ fails, the above profile decomposition can be applied to study a minimizing sequence of blow-up solutions to (NLW) with respect to the $L_t^\infty(\dot{H}_x^{s_c}\times \dot{H}_x^{s_c-1})$ norm.  Through this analysis, one may extract a minimal blow-up solution which is then shown to posess an additional compactness property up to the symmetries of the equation.  

More precisely, we recall the following result from \cite{KenigMerleSupercritical}.
\begin{theorem} \cite{KenigMerleSupercritical}
\label{l11}
Suppose that Theorem $\ref{l2}$ fails.  Then there exists a solution $u:I\times\mathbb{R}^d\rightarrow\mathbb{R}$ to (NLW) with maximal interval of existence $I$,
\begin{align*}
(u,u_t)\in L_t^\infty(I;\dot{H}_x^{s_c}\times\dot{H}_x^{s_c-1}),\quad\textrm{and}\quad \lVert u\rVert_{L_{t,x}^{d+1}(I\times\mathbb{R}^d)}=\infty
\end{align*}
such that $u$ is a minimal blow-up solution in the following sense: for any solution $v$ with maximal interval of existence $J$ such that $\lVert v\rVert_{L_{t,x}^{d+1}(J\times\mathbb{R}^d)}=\infty$, we have
\begin{align*}
\sup_{t\in I} \lVert (u(t),u_t(t))\rVert_{\dot{H}_x^{s_c}\times\dot{H}_x^{s_c-1}}&\leq \sup_{t\in J} \lVert (v(t),v_t(t))\rVert_{\dot{H}_x^{s_c}\times\dot{H}_x^{s_c-1}}.
\end{align*}

Moreover, there exist $N:I\rightarrow \mathbb{R}^+$ and $x:I\rightarrow\mathbb{R}^d$ such that the set
\begin{align}
K&=\{(\frac{1}{N(t)}u(t,x(t)+\frac{x}{N(t)}),\,\frac{1}{N(t)^{2}}u_t(t,x(t)+\frac{x}{N(t)})):t\in I\},\label{l12}
\end{align}
has compact closure in $\dot{H}_x^{s_c}\times\dot{H}_x^{s_c-1}(\mathbb{R}^d)$.
\end{theorem}

The above theorem was proved by Kenig and Merle in \cite{KenigMerleSupercritical} in three dimensions with radial initial data.  However, as pointed out in \cite{KenigLectureNotes,KenigLectureNotes2}, when a satisfactory local theory is present the proof is independent of the dimension and the assumption of radial symmetry.  We briefly summarize the main steps of the argument.  First, by means of the profile decomposition along with the local theory (local well-posedness and stability) discussed in Section 4 below, a minimal blow-up solution is extracted.  Then, the remainder of the proof consists of showing the compactness property ($\ref{l12}$), which is a consequence of the minimality.  For a detailed treatment, we refer the reader to the works \cite{KenigMerleNLW,KenigMerleSupercritical}.

\subsection{Three blow-up scenarios}
In view of Theorem $\ref{l11}$, if Theorem $\ref{l2}$ fails then there exists a minimal blow-up solution with the compactness property ($\ref{l12}$).  To obtain the desired contradiction, the next step in the argument is to show that no such blow-up solution can exist.  As we will see below, this failure of existence arises as a consequence of the compactness property ($\ref{l12}$). 
Before proceeding further, we now recall an equivalent formulation of ($\ref{l12}$) from \cite{KillipVisanSupercriticalNLW3D,KillipVisanSupercriticalNLWradial} which will be an essential tool for our analysis of blow-up solutions.
\begin{definition} A solution $u$ to (NLW) with time interval $I$ is said to be almost periodic modulo symmetries if $(u,u_t)\in L_t^\infty(I;\dot{H}_x^{s_c}\times\dot{H}_x^{s_c-1})$ and there exist functions $N:I\rightarrow\mathbb{R}^+$, $x:I\rightarrow\mathbb{R}^d$ and $C:\mathbb{R}^+\rightarrow\mathbb{R}^+$ such that for all $t\in I$ and $\eta>0$,
\begin{align*}
\int_{|x-x(t)|\geq C(\eta)/N(t)} ||\nabla|^{s_c}u(t,x)|^2+||\nabla|^{s_c-1}u_t(t,x)|^2dx&\leq \eta,
\end{align*}
and
\begin{align*}
\int_{|\xi|\geq C(\eta)N(t)} |\xi|^{2s_c}|\hat{u}(t,\xi)|^2+|\xi|^{2(s_c-1)}|\hat{u}_t(t,\xi)|^2d\xi&\leq \eta.
\end{align*}
\end{definition}
We will also record two consequences of almost periodicity from \cite{KillipVisanSupercriticalNLS,KillipVisanSupercriticalNLW3D}.
\begin{remark}
\label{l13}
If $u$ is an almost periodic solution modulo symmetries, then 
for each $\eta>0$ there exist constants $c_1(\eta),c_2(\eta)>0$ such that for all $t\in I$,
\begin{align*}
\int_{|x-x(t)|\geq c_1(\eta)/N(t)} |u(t,x)|^ddx+\int_{|x-x(t)|\geq c_1(\eta)/N(t)} |u_t(t,x)|^\frac{d}{2}dx\leq \eta.
\end{align*}
and also 
\begin{align}
\int_{|\xi|\leq c_2(\eta)N(t)} |\xi|^{2s_c}|\hat{u}(t,\xi)|^2+|\xi|^{2(s_c-1)}|\hat{u}_t(t,\xi)|^2d\xi\leq \eta.
\label{l14}
\end{align}
\end{remark}

The following theorem now shows that failure of Theorem $\ref{l2}$, in addition to implying the existence of a minimal blow-up solution (the consequence of Theorem $\ref{l11}$), also implies the existence of an almost periodic solution which belongs to one of three particular classes for which the associated function $N(t)$ is specified further.  Thus in order to prove Theorem $\ref{l2}$, it will suffice to show that such solutions cannot exist.
\begin{theorem} \cite{KillipVisanSupercriticalNLW3D}
\label{l15}
Suppose that Theorem $\ref{l2}$ fails.  Then there exists a solution $u:I\times\mathbb{R}^d\rightarrow\mathbb{R}$ to (NLW) with maximal interval of existence $I$ such that $u$ is almost periodic modulo symmetries,
\begin{align*}
(u,u_t)\in L_t^\infty(I;\dot{H}_x^{s_c}\times\dot{H}_x^{s_c-1}),\quad\textrm{and}\quad \lVert u\rVert_{L_{t,x}^{d+1}(I\times\mathbb{R}^d)}=\infty,
\end{align*}
and $u$ satisfies one of the following:
\begin{itemize}
\item (finite time blow-up solution) either $\sup I<\infty$ or $\inf I>-\infty$.
\item (soliton-like solution) $I=\mathbb{R}$ and $N(t)=1$ for all $t\in\mathbb{R}$.
\item (low-to-high frequency cascade solution) $I=\mathbb{R}$,
\begin{align*}
\inf_{t\in\mathbb{R}} N(t)\geq 1,\quad\textrm{and}\quad \limsup_{t\rightarrow\infty} N(t)=\infty.
\end{align*}
\end{itemize}
\label{l16}
\end{theorem}
In the context of the mass critical nonlinear Schr\"odinger equation, a more refined version of this theorem was proved by Killip, Tao and Visan in \cite{KillipTaoVisanCubic}.  The version that we use here was obtained by Killip and Visan in \cite{KillipVisanECritical}.  As remarked in \cite{KillipVisanSupercriticalNLW3D}, the argument applies equally to the present NLW setting.

\subsection{The contradiction}
We conclude our proof of Theorem $\ref{l2}$ by showing that each of the scenarios identified in Theorem $\ref{l15}$ cannot occur.  

The key ingredient that we use to rule out each of these scenarios is the conservation of energy.  However, we note that in our current setting we do not have immediate access to the finiteness of energy, since it has scaling below the critical regularity.  Nevertheless, in our analysis of each scenario, this obstruction is overcome with an observation that the solutions in that case do indeed have finite energy, due to the particular properties they possess.   We then exploit the conservation of energy in a manner well-suited to each scenario to obtain the desired contradiction.

We now briefly describe how we exclude each possible scenario in Theorem $\ref{l15}$:

We first consider the finite time blow-up solution.  In this case, our arguments are in the spirit of related results in \cite{KenigMerleNLW, KenigMerleSupercritical}.  We also note that a similar approach is taken in \cite{KillipVisanSupercriticalNLW3D}.  The key observation here is that when the maximal interval of existence of a solution $u$ is finite, the finite speed of propagation forces the supports of $u$ and $u_t$ to be localized to a ball which shrinks to $0$ as one approaches the blow-up time (see Lemma \ref{l49}).  We then show that the energy $E(u(t),u_t(t))$ tends to $0$ as $t$ tends to the blow-up time, contradicting the construction of $u$ as a blow-up solution.

We next study the remaining two scenarios, the soliton-like solution and the low-to-high frequency cascade.  In these cases, as in \cite{KillipVisanSupercriticalNLS,KillipVisanSupercriticalNLW3D}, we prove that the solutions possess an additional decay property: for almost periodic solutions with the function $N(t)$ bounded away from zero, the a priori bound $(u,u_t)\in L_t^\infty(\dot{H}_x^{s_c}\times \dot{H}_x^{s_c-1})$ allows us to obtain the bound $(u,u_t)\in L_t^\infty(\dot{H}_x^{1-\epsilon}\times \dot{H}_x^{-\epsilon})$ for some $\epsilon>0$ (see Theorem $\ref{l56}$ for further details).  In the NLS context the corresponding result was obtained in \cite{KillipVisanECritical,KillipVisanSupercriticalNLS},  while for the energy-supercritical NLW in $d=3$, see \cite{KillipVisanSupercriticalNLW3D}.

A main ingredient in the proof of the additional decay property is the following Duhamel formula, which states that if $u$ is an almost periodic solution, the linear components of the evolutions $u$ and $u_t$ vanish as $t$ approaches the endpoints of $I$.  In the context of the mass critical NLS, this formula was introduced in \cite{TaoVisanZhangMinimalMass} (see also \cite{KillipVisanClayNotes} for further discussion).  We recall the version that we use here from \cite{KillipVisanSupercriticalNLW3D}.
\begin{lemma} \cite{KillipVisanSupercriticalNLW3D,TaoVisanZhangMinimalMass}
\label{l17}
Let $u:I\times\mathbb{R}^d\rightarrow\mathbb{R}$ be a solution to (NLW) with maximal interval of existence $I$ which is almost periodic modulo symmetries.  Then for all $t\in I$,
\begin{align}
\nonumber&\bigg(\int_t^T \frac{\sin((t-t')|\nabla|)}{|\nabla|}F(u(t'))dt',\int_t^T \cos((t-t')|\nabla|)F(u(t'))dt'\bigg)\\
&\hspace{3in}\mathop{\rightharpoonup}_{T\rightarrow \sup I} (u(t),u_t(t)),
\label{l18}
\end{align}
and
\begin{align}
\nonumber &\bigg(-\int_T^t \frac{\sin((t-t')|\nabla|)}{|\nabla|}F(u(t'))dt',-\int_T^t \cos((t-t')|\nabla|)F(u(t'))dt'\bigg)\\
&\hspace{3in}\mathop{\rightharpoonup}_{T\rightarrow \inf I} (u(t),u_t(t)).
\end{align}
weakly in $\dot{H}_x^{s_c}\times \dot{H}_x^{s_c-1}$.
\end{lemma}

Arguing as in \cite{KillipVisanSupercriticalNLS,KillipVisanSupercriticalNLW3D}, we prove the additional decay property as follows:
\begin{itemize}
\item (Lemma $\ref{l58}$) We first refine the bound $u\in L_t^\infty L_x^d$ (which is immediate from the Sobolev embedding and the a priori assumption $u\in L_t^\infty \dot{H}_x^{s_c}$) to $L_t^\infty L_x^p$ for some $p<d$.  In particular, we use a bootstrap argument to bound the low frequencies of $u$ via Lemma $\ref{l17}$, while the high frequencies are bounded by the a priori bound.  We note that this argument imposes the restriction $p>2(d-1)/(d-3)$.
\item (Lemma $\ref{l60}$) We next use this $L_t^\infty L_x^p$ bound to improve bounds of the form $(u,u_t)\in L_t^\infty(\mathbb{R};\dot{H}_x^{s}\times\dot{H}_x^{s-1})$ to $(u,u_t)\in L_t^\infty(\mathbb{R};\dot{H}_x^{s-s_0}\times\dot{H}_x^{s-1-s_0})$ for some $s_0>0$.  This is accomplished by using the double Duhamel technique \cite{CKSTT,Tao}.  More precisely, we consider the inner product of the forward-in-time Duhamel formula with its backward-in-time counterpart given in Lemma $\ref{l17}$, and use the dispersive estimate.  When $p$ is such that the resulting integrals are convergent, this gives the desired improvement.  We note that this argument imposes the restriction $p<d-1$.
\item (Theorem $\ref{l56}$) Once we obtain the second step, we iterate the argument, starting with the a priori bound $(u,u_t)\in L_t^\infty (\dot{H}_x^{s_c}\times\dot{H}_x^{s_c-1})$, to obtain the desired decay $L_t^\infty(\dot{H}_x^{1-\epsilon}\times\dot{H}_x^{-\epsilon})$ for some $\epsilon>0$.  In particular, we obtain that the energy is finite.
\end{itemize}

We remark that the balance between the bounds provided by Lemma $\ref{l58}$ and the bound required by Lemma $\ref{l60}$ is the source of our restriction to dimensions $d\geq 6$.  As we noted above, Lemma $\ref{l58}$ provides the $L_t^\infty L_x^p$ bounds for $p>2(d-1)/(d-3)$, while Lemma $\ref{l60}$ requires this bound with $p<d-1$.  These conditions on $p$ impose the restriction $d\geq 6$.

We now return to the study of the two remaining blow-up scenarios: the soliton-like solution and the low-to-high frequency cascade solution.  

To preclude the soliton-like solution, we note that the finite speed of propagation implies a bound on the growth of $x(t)$ (see Lemma $\ref{l85}$), while the almost periodicity gives a uniform bound from below on the $L_t^4([s,s+1];L_x^4(\mathbb{R}^d))$ norm (see Lemma $\ref{l82}$).  The latter bound is closely related to a similar bound in \cite{KillipVisanSupercriticalNLW3D}.  However, we point out that in \cite{KillipVisanSupercriticalNLW3D} the bound is based on the $L_x^d$ norm, while our estimate is obtained via the $L_x^{2d/(d-2)}$ norm.  This allows us to use the dispersive estimate to control the linear propagator, rather than using the Strichartz estimate and a bootstrap argument.  Arguing as in \cite{KillipVisanSupercriticalNLW3D}, we then obtain a contradiction via the Morawetz identity by combining the bound on $x(t)$ with the $L_{t,x}^4$ bound and the finiteness of energy.

To conclude, as in the soliton-like solution, our preclusion of the low-to-high frequency cascade scenario is also based on the additional decay result.  We argue in a similar spirit as in \cite{KillipVisanSupercriticalNLS} to show that the energy tends to $0$ as $N(t)$ approaches infinity.  Since the energy is conserved, this contradicts our construction of $u$ as a blow-up solution.

\parskip=8pt
\section{Review of the local theory}

In this section, we review the standard local theory: local well-posedness and stability theorems for (NLW).  The versions that we present here are in the spirit of \cite{KenigMerleNLW,KenigMerleSupercritical,KillipVisanSupercriticalNLS,KillipVisanSupercriticalNLWradial,TaoVisan}.  

We note that the product structure of the cubic nonlinearity $F(u)=|u|^2u$ plays an important role in our arguments.  In particular, the necessary estimates on the nonlinearity reduce to the following product rule for fractional derivatives; see for instance \cite{ChristWeinstein,KenigPonceVega}.
\begin{lemma}\label{l19} 
For all $s\geq 0$ we have
\begin{align*}
\lVert |\nabla|^s(fg)\rVert_{L_x^p}\leq \lVert |\nabla|^s f\rVert_{L_x^{p_1}}\lVert g\rVert_{L_x^{p_2}}+\lVert f\rVert_{L_x^{p_3}}\lVert |\nabla|^s g\rVert_{L_x^{p_4}},
\end{align*}
where $1<p_1,p_4<\infty$ and $1<p,p_2,p_3\leq\infty$ satisfy $\frac{1}{p}=\frac{1}{p_1}+\frac{1}{p_2}=\frac{1}{p_3}+\frac{1}{p_4}$.
\end{lemma}

In the following two lemmas, using Lemma $\ref{l19}$, we obtain the estimates that will help us control the nonlinear term in establishing the local well-posedness and stability results.

\begin{lemma} Let $d\geq 6$ be given.  Then the following estimate holds:
\begin{align}
\nonumber \lVert |\nabla|^{\frac{d^2-4d+1}{2(d-1)}}(fg)\rVert_{N_{\frac{d-3}{2(d-1)}}}&\lesssim \lVert |\nabla|^\frac{d^2-4d+1}{2(d-1)} f\rVert_{S_{\frac{d+1}{2(d-1)}}}\lVert g\rVert_{L_{t,x}^\frac{d+1}{2}}\\
\nonumber &\hspace{0.2in}+\lVert |\nabla|^\frac{d^2-4d+1}{2(d-1)} f\rVert_{S_{\frac{d+1}{2(d-1)}}}\lVert |\nabla|^\frac{d^2-4d+1}{2(d-1)} g\rVert_{L_t^\frac{d+1}{2}L_x^\frac{2d(d^2-1)}{d^3+d^2-7d+1}}
\end{align}
\label{l20}
\end{lemma}

\begin{proof}
We begin by noting that $(\frac{2(d+1)}{d-3},\frac{2(d^2-1)}{d^2-2d+5})$ is an $\dot{H}_x^{\frac{d-3}{2(d-1)}}$ wave admissible pair. Applying Lemma $\ref{l19}$ followed by Sobolev's inequality, we obtain,
\begin{align}
&\nonumber \lVert |\nabla|^\frac{d^2-4d+1}{2(d-1)} (fg)\rVert_{N_{\frac{d-3}{2(d-1)}}}\lesssim \lVert |\nabla|^\frac{d^2-4d+1}{2(d-1)} (fg)\rVert_{L_t^\frac{2(d+1)}{d+5}L_x^\frac{2(d^2-1)}{d^2+2d-7}}\\
&\nonumber \hspace{0.7in}\lesssim \lVert |\nabla|^{\frac{d^2-4d+1}{2(d-1)}} f\rVert_{L_t^2L_x^\frac{2(d-1)}{d-3}}\lVert g\rVert_{L_{t,x}^\frac{d+1}{2}}\\
&\nonumber \hspace{0.9in}+\lVert f\rVert_{L_t^2L_x^{2d}}\lVert |\nabla|^\frac{d^2-4d+1}{2(d-1)}g\rVert_{L_t^\frac{d+1}{2}L_x^\frac{2d(d^2-1)}{d^3+d^2-7d+1}}\\
&\nonumber \hspace{0.7in}\lesssim \lVert |\nabla|^{\frac{d^2-4d+1}{2(d-1)}} f\rVert_{L_t^2L_x^\frac{2(d-1)}{d-3}}\lVert g\rVert_{L_{t,x}^\frac{d+1}{2}} \\
&\nonumber \hspace{0.9in}+\lVert |\nabla|^{\frac{d^2-4d+1}{2(d-1)}}f\rVert_{L_t^2L_x^\frac{2(d-1)}{d-3}}\lVert |\nabla|^\frac{d^2-4d+1}{2(d-1)}g\rVert_{L_t^\frac{d+1}{2}L_x^\frac{2d(d^2-1)}{d^3+d^2-7d+1}}.
\end{align}

We conclude the proof by noting that $(2,\frac{2(d-1)}{d-3})$ is an $\dot{H}_x^\frac{d+1}{2(d-1)}$ admissible pair, which gives the right hand side of the desired inequality.
\end{proof}

We will also need the following estimate, which is a variant of the fractional chain rule for the cubic nonlinearity.
\begin{lemma} Let $d\geq 6$ be given.  Then we have,
\begin{align*}
\nonumber \lVert |\nabla|^{\frac{d^2-4d+1}{2(d-1)}}(|f|^2f)\rVert_{N_{\frac{d-3}{2(d-1)}}}&\lesssim \lVert |\nabla|^\frac{d^2-4d+1}{2(d-1)} f\rVert_{L_t^2L_x^\frac{2(d-1)}{d-3}}\lVert f\rVert^2_{L_{t,x}^{d+1}}\\
&\lesssim \lVert |\nabla|^\frac{d^2-4d+1}{2(d-1)} f\rVert_{S_{\frac{d+1}{2(d-1)}}}\lVert f\rVert^2_{L_{t,x}^{d+1}}.
\end{align*}
\label{l21}
\end{lemma}
\begin{proof}
We note that, proceeding as in the proof of Lemma $\ref{l20}$,
\begin{align}
\nonumber &\lVert |\nabla|^\frac{d^2-4d+1}{2(d-1)} (|f|^2f)\rVert_{N_{\frac{d-3}{2(d-1)}}}\lesssim \lVert |\nabla|^\frac{d^2-4d+1}{2(d-1)} (|f|^2f)\rVert_{L_t^\frac{2(d+1)}{d+5}L_x^\frac{2(d^2-1)}{d^2+2d-7}}\\
\nonumber &\hspace{0.2in}\lesssim \lVert |\nabla|^{\frac{d^2-4d+1}{2(d-1)}} f\rVert_{L_t^2L_x^\frac{2(d-1)}{d-3}}\lVert f^2\rVert_{L_{t,x}^\frac{d+1}{2}} \\
\nonumber &\hspace{0.4in}+\lVert f\rVert_{L_{t,x}^{d+1}}\lVert |\nabla|^{\frac{d^2-4d+1}{2(d-1)}} (f^2)\rVert_{L_t^\frac{2(d+1)}{d+3}L_x^\frac{2(d^2-1)}{d^2-5}}\\
\nonumber &\hspace{0.2in}\lesssim \lVert |\nabla|^{\frac{d^2-4d+1}{2(d-1)}} f\rVert_{L_t^2L_x^\frac{2(d-1)}{d-3}}\lVert f\rVert_{L_{t,x}^{d+1}}^2\\
\nonumber &\hspace{0.4in}+\lVert f\rVert_{L_{t,x}^{d+1}} \lVert |\nabla|^{\frac{d^2-4d+1}{2(d-1)}}f\rVert_{L_t^2L_x^\frac{2(d-1)}{d-3}}\lVert f\rVert_{L_{t,x}^{d+1}}\\
&\hspace{0.2in}\lesssim \lVert |\nabla|^\frac{d^2-4d+1}{2(d-1)} f\rVert_{S_{\frac{d+1}{2(d-1)}}}\lVert f\rVert_{L_{t,x}^{d+1}}^2,
\end{align}
where in the third inequality we use Lemma $\ref{l19}$ and we note that $(2,\frac{2(d-1)}{d-3})$ is an $\dot{H}_x^{\frac{d+1}{2(d-1)}}$-wave admissible pair to obtain the desired estimate.
\end{proof}

\subsection{Local well-posedness}

We now give a standard local well-posedness theorem for (NLW) with our cubic nonlinearity $F(u)=|u|^2u$.  The version that we present here is in the spirit of the related results in the works of \cite{KenigMerleNLW,KillipVisanSupercriticalNLS}.  For similar results see also \cite{BulutCzubakLiPavlovicZhang,GinibreSoferVelo,KenigMerleNLS,Pecher,ShatahStruwe2,TaoVisan}.
\begin{theorem}
\label{l22}
Let $d\geq 6$ and $s_c=\frac{d-2}{2}$.  Then for all $A>0$, there exists $\delta_0=\delta_0(d,A)>0$ such that for every $0<\delta\leq \delta_0$, $0\in I\subset\mathbb{R}$, and $(u_0,u_1)\in \dot{H}_x^{s_c}\times \dot{H}_x^{s_c-1}(\mathbb{R}^d)$ with 
\begin{align}
\lVert (u_0,u_1)\rVert_{\dot{H}_x^{s_c}\times\dot{H}_x^{s_c-1}}\leq A,
\label{l23}
\end{align}
the condition
\begin{align*}
\lVert \mathcal{W}(t)(u_0,u_1)\rVert_{L_{t,x}^{d+1}(I\times\mathbb{R}^d)}\leq \delta,
\end{align*}
implies that there exists a unique solution $u$ to (NLW) on $I\times\mathbb{R}^d$ with
\begin{align*}
\lVert u\rVert_{L_{t,x}^{d+1}}\leq 2\delta,
\end{align*}
and
\begin{align*}
\lVert |\nabla|^{\frac{d^2-4d+1}{2(d-1)}}u\rVert_{S_{\frac{d+1}{2(d-1)}}(I)}+\lVert |\nabla|^{\frac{d^2-4d+1}{2(d-1)}-1}u_t\rVert_{S_{\frac{d+1}{2(d-1)}}(I)}<\infty.
\end{align*}
\end{theorem}

\begin{proof}
We use a contraction mapping argument.  Fix $\alpha=\frac{d^2-4d+1}{2(d-1)}$ and note that by the Duhamel representation for the solution to (NLW), we have
\begin{align*}
u(t)=\mathcal{W}(t)(u_0,u_1)+\int_0^t \frac{\sin((t-s)|\nabla|)}{|\nabla|}(|u(s)|^2u(s))ds.
\end{align*}
For all $a,b>0$, we define the contraction space
\begin{align*}
B_{a,b}&:=\{v:\lVert v\rVert_{L_{t,x}^{d+1}}\leq a, \\
&\hspace{1.0in}\lVert |\nabla|^\alpha v\rVert_{S_{\frac{d+1}{2(d-1)}}}+\lVert |\nabla|^{\alpha-1}v_t\rVert_{S_{\frac{d+1}{2(d-1)}}}\leq b\},
\end{align*}
and the map
\begin{align*}
\Phi(v)(t):=\mathcal{W}(t)(u_0,u_1)+\int_0^t \frac{\sin((t-s)|\nabla|)}{|\nabla|} (|v(s)|^2v(s))ds.
\end{align*}

We would like to show that for suitably chosen $a$ and $b$, we have the inclusion $\Phi(B_{a,b})\subset B_{a,b}$ and the mapping $\Phi:B_{a,b}\rightarrow B_{a,b}$ is a contraction.

We first note that using Minkowski's inequality followed by the assumption ($\ref{l23}$) and the Strichartz inequality, we obtain for $v\in B_{a,b}$,
\begin{align}
&\nonumber \lVert |\nabla|^\alpha \Phi(v)\rVert_{S_{s_c-\alpha}}+\lVert |\nabla|^{\alpha-1}\partial_t \Phi(v)\rVert_{S_{s_c-\alpha}}\\
&\nonumber \hspace{0.2in}\leq \lVert |\nabla|^\alpha \mathcal{W}(t)(u_0,u_1)\rVert_{S_{s_c-\alpha}}+\bigg\lVert |\nabla|^\alpha \int_0^t \frac{\sin((t-s)|\nabla|)}{|\nabla|} (|v(s)|^2v(s))ds\bigg\rVert_{S_{s_c-\alpha}}\\
&\nonumber \hspace{0.4in}+\lVert |\nabla|^\alpha \partial_t \mathcal{W}(t)(u_0,,u_1)\rVert_{S_{s_c-\alpha}}+\bigg\lVert |\nabla|^{\alpha-1}\int_0^t \cos((t-s)|\nabla|)(|v(s)|^2v(s))ds\bigg\rVert_{S_{s_c-\alpha}}\\
&\nonumber \hspace{0.2in}\lesssim \lVert (u_0,u_1)\rVert_{\dot{H}_x^{s_c}\times\dot{H}_x^{s_c-1}}+\lVert |\nabla|^\alpha (|v|^2v)\rVert_{N_{1+\alpha-s_c}}\\
&\label{l24}\hspace{0.2in}\leq CA+C'\lVert |\nabla|^\alpha v\rVert_{S_{s_c-\alpha}}\lVert v\rVert_{L_{t,x}^{d+1}}^2\\
&\nonumber\hspace{0.2in}\leq CA+Ca^2b,
\end{align}
where we used Lemma \ref{l21} to obtain ($\ref{l24}$).

\noindent Similarly, using Minkowski's inequality together with the assumption ($\ref{l23}$), we estimate
\begin{align*}
\lVert \Phi(v)\rVert_{L_{t,x}^{d+1}}&\leq \lVert \mathcal{W}(t)(u_0,u_1)\rVert_{L_{t,x}^{d+1}}+\bigg\lVert \int_0^t \frac{\sin((t-s)|\nabla|)}{|\nabla|}(|u(s)|^2u(s))ds\bigg\rVert_{L_{t,x}^{d+1}}\\
&\leq \delta+C\lVert |\nabla|^\alpha (|u|^2u)\rVert_{N_{1+\alpha-s_c}}\\
&\leq \delta+C\lVert |\nabla|^\alpha u\rVert_{S_{s_c-\alpha}}\lVert u\rVert_{L_{t,x}^{d+1}}^2\\
&\leq \delta+Ca^2b.
\end{align*}

\noindent Choosing $b=2AC$ and $a$ such that $Ca^2\leq \frac{1}{2}$, we obtain
\begin{align}
\lVert |\nabla|^\alpha \Phi(v)\rVert_{S_{s_c-\alpha}}\leq b.\label{l25}
\end{align}
If we also fix $\delta=\frac{a}{2}$ and $a$ small enough such that $Ca^2b\leq \frac{a}{2}$, we have
\begin{align}
\lVert \Phi(v)\rVert_{L_{t,x}^{d+1}}\leq a.\label{l26}
\end{align}
Combining ($\ref{l25}$) and ($\ref{l26}$) with the above choices of $a,b$ and $\delta$, we have the desired inclusion $\Phi(B_{a,b})\subset B_{a,b}$.

We now show that the mapping $\Phi$ is a contraction for suitable $a,b$ and $\delta$.  Let $a,b$ and $\delta$ be as chosen above.  Note that by the Strichartz inequality and Lemma $\ref{l20}$ along with Minkowski's inequality we have, 
\begin{align}
\nonumber &\lVert |\nabla|^\alpha [\Phi(u)-\Phi(v)]\rVert_{S_{s_c-\alpha}}+\lVert |\nabla|^{\alpha-1}\partial_t [\Phi(u)-\Phi(v)]\rVert_{S_{s_c-\alpha}}+\lVert \Phi(u)-\Phi(v)\rVert_{L_{t,x}^{d+1}}\\
\nonumber &\hspace{0.2in}\lesssim \lVert |\nabla|^\alpha [(|v|^2v)-(|u|^2u)]\rVert_{N_{1+\alpha-s_c}}\\
\nonumber &\hspace{0.2in}=\lVert |\nabla|^\alpha [(v-u)\{v^2+uv+u^2\}]\rVert_{N_{\frac{d-3}{2(d-1)}}}\\
\nonumber &\hspace{0.2in}\leq \lVert |\nabla|^\alpha (v-u)\rVert_{S_{\frac{d+1}{2(d-1)}}}\\
\nonumber &\hspace{0.5in}\big[\lVert \{v^2+uv+u^2\}\rVert_{L_{t,x}^\frac{d+1}{2}}+\lVert |\nabla|^\alpha \{v^2+uv+u^2\}\rVert_{L_t^\frac{d+1}{2}L_x^\frac{2d(d^2-1)}{d^3+d^2-7d+1}}\big]\\
\nonumber &\hspace{0.2in}\leq \lVert |\nabla|^\alpha (v-u)\rVert_{S_{\frac{d+1}{2(d-1)}}}\\
\nonumber &\hspace{0.5in}\big[\lVert v^2\rVert_{L_{t,x}^\frac{d+1}{2}}+\lVert uv\rVert_{L_{t,x}^\frac{d+1}{2}}+\lVert u^2\rVert_{L_{t,x}^\frac{d+1}{2}}+\lVert |\nabla|^\alpha (v^2)\rVert_{L_t^\frac{d+1}{2}L_x^\frac{2d(d^2-1)}{d^3+d^2-7d+1}}\\
\nonumber &\hspace{0.5in}+\lVert |\nabla|^\alpha (uv)\rVert_{L_t^\frac{d+1}{2}L_x^\frac{2d(d^2-1)}{d^3+d^2-7d+1}}+\lVert |\nabla|^\alpha (u^2)\rVert_{L_t^\frac{d+1}{2}L_x^\frac{2d(d^2-1)}{d^3+d^2-7d+1}}\big]\\
\label{l27}&\hspace{0.2in}\lesssim \lVert v-u\rVert_{B_{a,b}}\\
\nonumber &\hspace{0.5in}\big[\lVert v\rVert_{L_{t,x}^{d+1}}^2+\lVert u\rVert_{L_{t,x}^{d+1}}\lVert v\rVert_{L_{t,x}^{d+1}}+\lVert u\rVert_{L_{t,x}^{d+1}}^2\\
\nonumber &\hspace{0.5in}+\lVert |\nabla|^\alpha v\rVert_{L_t^{d+1}L_x^{\frac{2d(d^2-1)}{d^3-d^2-5d+1}}}\lVert v\rVert_{L_{t,x}^{d+1}}+\lVert |\nabla|^\alpha u\rVert_{L_t^{d+1}L_x^{\frac{2d(d^2-1)}{d^3-d^2-5d+1}}}\lVert v\rVert_{L_{t,x}^{d+1}}\\
\nonumber &\hspace{0.5in}+\lVert |\nabla|^\alpha v\rVert_{L_{t,x}^{d+1}}\lVert v\rVert_{L_{t,x}^{d+1}}+\lVert |\nabla|^\alpha u\rVert_{L_t^{d+1}L_x^{\frac{2d(d^2-1)}{d^3-d^2-5d+1}}}\lVert u\rVert_{L_{t,x}^{d+1}}\big]\\
\nonumber &\hspace{0.2in}\lesssim \lVert u-v\rVert_{B_{a,b}}(a^2+ab),
\end{align}
where we use H\"older's inequality and Lemma $\ref{l19}$ to obtain ($\ref{l27}$).  Thus, if $a$ is chosen such that $C(a^2+ab)<1$ we conclude that $\Phi$ is a contraction as desired.
\end{proof}

\begin{remark}
Note that if $u^{(1)}$ and $u^{(2)}$ are two solutions to (NLW) as stated in Section $1$ with maximal interval of existence $I$ such that $(u^{(1)}(0),u^{(1)}_t(0))=(u^{(2)}(0),u^{(2)}_t(0))$, then
\begin{align*}
u^{(1)}(t)=u^{(2)}(t)\quad\textrm{for all}\quad t\in I.
\end{align*}
This result follows from standard arguments; see for instance \cite[\S IV.3]{Sogge}.
\end{remark}

\subsection{Stability}

In this section, we prove a stability result for (NLW).  As in the local well-posedness theorem, the argument that we present follows a standard approach and makes use of the cubic nature of the nonlinearity $F(u)=|u|^2u$.  In particular, the argument that we present here is in the spirit of the related works \cite{KenigMerleSupercritical,KillipVisanSupercriticalNLS}.  For similar treatments, see also \cite{BulutCzubakLiPavlovicZhang,KenigLectureNotes2,KillipVisanSupercriticalNLWradial,TaoVisan}.
\begin{theorem}
Let $d\geq 6$ and $s_c=\frac{d-2}{2}$.  Assume $0\in I\subset\mathbb{R}$ is a compact time interval and $\tilde{u}:I\times\mathbb{R}^d\rightarrow\mathbb{R}$ is a solution of the equation
\begin{align*}
\tilde{u}_{tt}-\Delta\tilde{u}+|\tilde{u}|^2\tilde{u}=e,
\end{align*}
for some $e$.  

Then for every $E,L>0$, there exists $\epsilon_1=\epsilon_1(E,L)>0$ such that for each $0<\epsilon<\epsilon_1$, the conditions
\begin{align*}
\sup_{t\in I} \lVert (\tilde{u}(t),\tilde{u}_t(t))\rVert_{\dot{H}_x^{s_c}\times\dot{H}_x^{s_c-1}(\mathbb{R}^d)}&\leq E,
\end{align*}\begin{align*}
\lVert (u_0-\tilde{u}(0),u_1-\tilde{u}_t(0))\rVert_{\dot{H}_x^{s_c}\times\dot{H}_x^{s_c-1}(\mathbb{R}^d)}&\leq \epsilon,
\end{align*}
\begin{align*}
\lVert |\nabla|^\frac{d^2-4d+1}{2(d-1)}e\rVert_{N_{\frac{d-2}{2(d-1)}}(I)}&\leq\epsilon, \quad\textrm{and}
\end{align*}\begin{align*}
\lVert \tilde{u}\rVert_{L_{t,x}^{d+1}}&\leq L
\end{align*}
imply that there exists a unique solution $u:I\times\mathbb{R}^d\rightarrow\mathbb{R}$ to (NLW) with initial data $(u_0,u_1)$ such that
\begin{align}
\lVert \tilde{u}-u\rVert_{L_{t,x}^{d+1}}&\leq C(E,L)\epsilon,\label{l28}
\end{align}\begin{align}
\lVert |\nabla|^{\frac{d^2-4d+1}{2(d-1)}}(u-\tilde{u})\rVert_{S_{\frac{d+1}{2(d-1)}}(I)}&\leq C(E,L)\epsilon,\label{l29}
\end{align}\begin{align}
\lVert |\nabla|^{\frac{d^2-4d+1}{2(d-1)}}u\rVert_{S_{\frac{d+1}{2(d-1)}}(I)}&\leq C(E,L).\label{l30}
\end{align}
\end{theorem}

\begin{proof}
Fix $\alpha=\frac{d^2-4d+1}{2(d-1)}$.  We begin by obtaining a bound on
\begin{align*}
\lVert |\nabla|^{\alpha}\tilde{u}\rVert_{S_{s_c-\alpha}(I)}.
\end{align*}
To do so, we fix $\epsilon_1, \eta>0$ (to be determined later in the argument) and partition $I$ into $J_0=J_0(L,\eta)$ subintervals $I_j=[t_j,t_{j+1}]$ such that for each $j=1,\cdots, J_0$,
\begin{align*}
\lVert \tilde{u}\rVert_{L_{t,x}^{d+1}(I_j\times\mathbb{R}^d)}&\leq \eta.
\end{align*}
Applying the Strichartz inequality followed by Lemma $\ref{l21}$, we obtain 
\begin{align*}
\lVert |\nabla|^\alpha\tilde{u}\rVert_{S_{s_c-\alpha}(I_j)}&\lesssim \lVert (\tilde{u}(t_j),\tilde{u}_t(t_j))\rVert_{\dot{H}_x^{s_c}\times \dot{H}_x^{s_c-1}}\\
&\hspace{0.2in}+\lVert |\nabla|^\alpha e\rVert_{N_{1+\alpha-s_c}(I_j)}+\lVert |\nabla|^\alpha F(\tilde{u}(s))\rVert_{N_{1+\alpha-s_c}(I_j)}\\
&\lesssim E+\epsilon+\lVert \tilde{u} \rVert_{L_{t,x}^{d+1}}^2\lVert |\nabla|^\alpha\tilde{u}\rVert_{S_{s_c-\alpha}(I_j)}\\
&\lesssim E+\epsilon_1+\eta^2\lVert |\nabla|^\alpha\tilde{u}\rVert_{S_{s_c-\alpha}(I_j)} 
\end{align*}
for each $\epsilon<\epsilon_1$.  Choosing $\eta>0$ sufficiently small and $\epsilon_1<E$, we obtain 
\begin{align*}
\lVert |\nabla|^\alpha \tilde{u}\rVert_{S_{s_c-\alpha}(I_j)}\lesssim E.
\end{align*}
Summing the contributions of the subintervals, we conclude 
\begin{align}
\lVert |\nabla|^\alpha \tilde{u}\rVert_{S_{s_c-\alpha}(I)}\lesssim C(E,L).\label{l31}
\end{align}
as desired.

To continue, fixing $\epsilon_1\leq E$ and $\delta>0$ (to be determined later in the argument), we note that $(d+1,\frac{2d(d^2-1)}{d^3-d^2-5d+1})$ is an $\dot{H}_x^{\frac{d+1}{2(d-1)}}$-wave admissible pair.  Then by virtue of ($\ref{l31}$), we may divide $I$ into $J_1=J_1(E,L,\delta)$ subintervals $I_j=[t_j,t_{j+1}]$ such that for each $j=1,\cdots, J_1$, we have
\begin{align*}
\lVert |\nabla|^\alpha\tilde{u}\rVert_{L_t^{d+1}L_x^\frac{2d(d^2-1)}{d^3-d^2-5d+1}}&\leq \delta.
\end{align*}

Let $w=u-\tilde{u}$, and define, for $t\in I$ and $j=1,\cdots, J_1$,
\begin{align*}
\gamma_j(t):=\lVert |\nabla|^\alpha [F(\tilde{u}+w)-F(\tilde{u})]\rVert_{N_{1+\alpha-s_c}([t_j,t])}.
\end{align*}

\noindent Let $j\in \{1,\cdots,J_1\}$ be given.  We now obtain an estimate on $\gamma_j(t)$.  We begin by writing 
\begin{align*}
F(x)-F(y)=(x-y)[(x-y)^2+3xy].
\end{align*}
Invoking Lemma $\ref{l20}$, followed by Minkowski's and H\"older's inequalities, we obtain
\begin{align}
\nonumber \gamma_j(t)&\leq \lVert |\nabla|^\alpha w\rVert_{S_{s_c-\alpha}}\\
\nonumber &\hspace{0.2in}\big[\lVert w^2+3(\tilde{u}+w)\tilde u\rVert_{L_{t,x}^\frac{d+1}{2}}+\lVert |\nabla|^\alpha [w^2+3(\tilde{u}+w)\tilde{u}]\rVert_{L_t^\frac{d+1}{2}L_x^\frac{2d(d^2-1)}{d^3+d^2-7d+1}}\big]\\
\nonumber &\lesssim \lVert |\nabla|^\alpha w\rVert_{S_{s_c-\alpha}}\big[\lVert w^2\rVert_{L_{t,x}^\frac{d+1}{2}}+\lVert \tilde{u}^2\rVert_{L_{t,x}^\frac{d+1}{2}} + \lVert w\tilde u\rVert_{L_{t,x}^\frac{d+1}{2}}\\
\nonumber &\hspace{0.2in}+\lVert |\nabla|^\alpha [w^2]\rVert_{L_t^\frac{d+1}{2}L_x^\frac{2d(d^2-1)}{d^3+d^2-7d+1}}+\lVert |\nabla|^\alpha [\tilde{u}^2]\rVert_{L_t^\frac{d+1}{2}L_x^\frac{2d(d^2-1)}{d^3+d^2-7d+1}}\\
\nonumber &\hspace{0.2in}+\lVert |\nabla|^\alpha [w\tilde{u}]\rVert_{L_t^\frac{d+1}{2}L_x^\frac{2d(d^2-1)}{d^3+d^2-7d+1}}\big]\\
\nonumber &\lesssim \lVert |\nabla|^\alpha w\rVert_{S_{s_c-\alpha}}\big[\lVert w\rVert_{L_{t,x}^{d+1}}^2 + \lVert \tilde{u}\rVert_{L_{t,x}^{d+1}   }^2 + \lVert \tilde{u}\rVert_{L_{t,x}^{d+1}   } \lVert w\rVert_{L_{t,x}^{d+1}   } \\
\nonumber &\hspace{0.2in}+\lVert |\nabla|^\alpha w\rVert_{L_t^{d+1}L_x^\frac{2d(d^2-1)}{d^3-d^2-5d+1}} \lVert w\rVert_{L_{t,x}^{d+1}   }+ \lVert |\nabla|^\alpha \tilde{u} \rVert_{L_t^{d+1}L_x^\frac{2d(d^2-1)}{d^3-d^2-5d+1}} \lVert \tilde{u}\rVert_{L_{t,x}^{d+1}   }\\
\nonumber &\hspace{0.2in}+\lVert |\nabla|^\alpha w\rVert_{L_t^{d+1}L_x^\frac{2d(d^2-1)}{d^3-d^2-5d+1}} \lVert \tilde{u}\rVert_{L_{t,x}^{d+1}   }+ \lVert w\rVert_{L_{t,x}^{d+1}   }\lVert |\nabla|^\alpha \tilde{u}\rVert_{L_t^{d+1}L_x^\frac{2d(d^2-1)}{d^3-d^2-5d+1}}\big]\nonumber\\
\label{l32} &\lesssim \lVert |\nabla|^\alpha w\rVert_{S_{s_c-\alpha}(I_j)}^3+\delta\lVert |\nabla|^\alpha w\rVert_{S_{s_c-\alpha}(I_j)}^2+\delta^2\lVert |\nabla|^\alpha w\rVert_{S_{s_c-\alpha}(I_j)}.
\end{align}
where we have used Lemma $\ref{l19}$ along with Sobolev's inequality in obtaining the last inequality.

Having obtained the bound ($\ref{l32}$) on $\gamma_j(t)$ for all $j\in \{1,\cdots, J_1\}$, we next show by induction that for every $j=1,\cdots, J_1$, there exists a constant $C(j,d)>0$ such that
\begin{align}
\gamma_j(t)\leq C(j,d)\epsilon.\label{l33}
\end{align}

In the remainder of the argument, we let $\epsilon\in\mathbb{R}$ be arbitrary such that $\epsilon<\epsilon_1$ and we note that without loss of generality we may assume $t_1=0$.  

To obtain $(\ref{l33})$ we argue as follows: we first observe that when $j=1$, the Strichartz inequality gives, for every $t\in I_1$,
\begin{align}
\nonumber \lVert |\nabla|^\alpha w\rVert_{S_{s_c-\alpha}([t_1,t])}&\lesssim \lVert (w(t_1),w_t(t_1))\rVert_{\dot{H}_x^{s_c}\times \dot{H}_x^{s_c-1}}\\
\nonumber &\hspace{0.2in}+\lVert |\nabla|^\alpha [F(\tilde{u})-F(u)]\rVert_{N_{1+\alpha-s_c}([t_1,t])}+\lVert |\nabla|^\alpha e\rVert_{N_{1+\alpha-s_c }(I_1)}\\
\nonumber &\lesssim \lVert (w(0),w_t(0))\rVert_{\dot{H}_x^{s_c}\times \dot{H}_x^{s_c-1}}+\gamma_1(t)+\epsilon\\
\label{l34}&\lesssim \epsilon+\gamma_1(t)+\epsilon.
\end{align}
Putting ($\ref{l32}$) and ($\ref{l34}$) together, we obtain
\begin{align*}
\gamma_1(t)&\lesssim (\gamma_1(t)+\epsilon)^3+\delta(\gamma_1(t)+\epsilon)^2+\delta^2(\gamma_1(t)+\epsilon).
\end{align*}
A bootstrap argument then implies that for $\delta$ and $\epsilon$ sufficiently small, $\gamma_1(t)\lesssim \epsilon$ for all $t\in I_1$.

For the induction step, we now assume that for all $j\leq j_0$ there exists $C(j,d,\delta)>0$ such that $\gamma_j(t)\leq C(j,d)\epsilon$ for all $t\in I_j$.  We then prove the validity of ($\ref{l33}$) for $j=j_0+1$.  

\noindent Note that for every $t\in I_{j_0+1}$, two successive applications of the Strichartz inequality give
\begin{align}
\nonumber &\lVert |\nabla|^\alpha w\rVert_{S_{s_c-\alpha}([t_{j_0+1},t])}\lesssim \lVert (w(t_{j_0+1}),w_t(t_{j_0+1}))\rVert_{\dot{H}_x^{s_c}\times \dot{H}_x^{s_c-1}}\\
\nonumber &\hspace{0.4in}+\lVert |\nabla|^\alpha [F(\tilde{u})-F(u)]\rVert_{N_{1+\alpha-s_c}([t_{j_0+1},t])}+\lVert |\nabla|^\alpha e\rVert_{N_{1+\alpha-s_c }(I_{j_0+1})}\\
\nonumber &\hspace{0.2in}\lesssim \lVert (w(t_{j_0+1}),w_t(t_{j_0+1}))\rVert_{\dot{H}_x^{s_c}\times \dot{H}_x^{s_c-1}}+\gamma_{j_0+1}(t)+\epsilon\\
\nonumber &\hspace{0.2in}\lesssim \lVert (w(0),w_t(0))\rVert_{\dot{H}_x^{s_c}\times\dot{H}_x^{s_c-1}}+\lVert |\nabla|^\alpha [F(\tilde{u})-F(u)]\rVert_{N_{1+\alpha-s_c}([0,t_{j_0+1}])}\\
\nonumber &\hspace{0.4in}+\lVert |\nabla|^{\alpha}e\rVert_{N_{1+\alpha-s_c}([0,t_{j_0+1}])}+\gamma_{j_0+1}(t)+\epsilon\\
\nonumber &\hspace{0.2in}\lesssim 3\epsilon+\gamma_{j_0+1}(t)+\sum_{k=1}^{j_0} \gamma_k(t_{k+1})\\
\label{l35} &\hspace{0.2in}\lesssim \bigg(3+\sum_{k=1}^{j_0}C(k,d)\bigg)\epsilon+\gamma_{j_0+1}(t)
\end{align}
where we used the induction assumption in obtaining the last inequality.  Noting $\sum_{k=1}^{j_0} C(k,d)\lesssim C(j_0,d)$ and combining ($\ref{l32}$) and $(\ref{l35})$, we obtain
\begin{align*}
\gamma_{j_0+1}(t)&\lesssim (\gamma_{j_0+1}(t)+\epsilon)^3+\delta(\gamma_{j_0+1}(t)+\epsilon)^2+\delta^2(\gamma_{j_0+1}(t)+\epsilon).
\end{align*}
A bootstrap argument then implies that for $\delta$ and $\epsilon_1$ sufficiently small, $\gamma_{j_0+1}(t)\lesssim \epsilon$ for all $t\in I_{j_0+1}$.  This immediately establishes the inductive step $j_0\rightarrow j_0+1$.

Combining the estimates $(\ref{l33})$ that we have obtained on $\gamma_j(t)$ for $j=1,\cdots, J_1$, we obtain
\begin{align}
\lVert |\nabla|^\alpha [F(u)-F(\tilde u)]\rVert_{N_{1+\alpha-s_c}(I)}&\lesssim \sum_{j=1}^{J_1} \gamma_j(t_{j+1})\lesssim C(E,L)\epsilon
\label{l36}
\end{align}
where we note that $J_1=J_1(E,L)$.

We now conclude the proof by showing the desired bounds ($\ref{l28}$)-($\ref{l30}$).  For ($\ref{l28}$), we note that by the Sobolev embedding and the definition of the $S_{s_c-\alpha}$ norm, we have
\begin{align*}
\lVert \tilde{u}-u\rVert_{L_{t,x}^{d+1}}&\lesssim \lVert |\nabla|^\alpha(\tilde{u}-u)\rVert_{L_t^{d+1}L_x^{\frac{2d(d^2-1)}{d^3-d^2-5d+1}}}\\
&\lesssim \lVert |\nabla|^\alpha (\tilde{u}-u)\rVert_{S_{s_c-\alpha}}.
\end{align*}
On the other hand, for ($\ref{l30}$), Minkowski's inequality and ($\ref{l31}$) imply
\begin{align*}
\lVert |\nabla|^\alpha u\rVert_{S_{s_c-\alpha}}&\leq \lVert |\nabla|^\alpha (u-\tilde{u})\rVert_{S_{s_c-\alpha}}+\lVert |\nabla|^\alpha \tilde{u}\rVert_{S_{s_c-\alpha}}\\
&\lesssim \lVert |\nabla|^\alpha (\tilde{u}-u)\rVert_{S_{s_c-\alpha}}+C(E,L).
\end{align*}
Thus, both ($\ref{l28}$) and ($\ref{l30}$) follow from ($\ref{l29}$), which is proved as follows: by the Strichartz inequality and $(\ref{l36})$, we have
\begin{align*}
\lVert |\nabla|^\alpha (\tilde{u}-u)\rVert_{S_{s_c-\alpha}}&\lesssim \epsilon+\lVert |\nabla|^\alpha F(\tilde{u})-F(u)\rVert_{N_{1+\alpha-s_c}}\\
&\lesssim C(E,L)\epsilon.
\end{align*}
\end{proof}

\section{Finite speed of propagation.}

A key property of NLW which is not present in the NLS setting is the finite speed of propagation.  Using this property, we next give the following lemma which will facilitate our arguments in the proofs of Lemma $\ref{l49}$ and Lemma $\ref{l85}$.

Let $\psi$ be a smooth radial function such that $0\leq \psi\leq 1$ and
\begin{align*}
\psi(x)=\left\lbrace\begin{array}{l}0,\quad |x|<1,\\
1,\quad |x|\geq 2.\end{array}\right.
\end{align*}

For all $R>0$, define $\psi_R\in C^\infty(\mathbb{R}^d)$ by
\begin{align*}
\psi_R(x)=\psi(\frac{x}{R}),\quad x\in\mathbb{R}^d.
\end{align*}

\begin{lemma}
\label{l37}
Suppose that $u:I\times\mathbb{R}^d\rightarrow\mathbb{R}$ is an almost periodic solution to (NLW) with maximal interval of existence $I$ and $(u,u_t)\in L_t^\infty(I;\dot{H}_x^{s_c}\times\dot{H}_x^{s_c-1})$.

Then for each $\epsilon>0$ there exists $R>0$ such that for every $t\in I$, if $(v^{(t)}_0,v^{(t)}_1)$ is defined by
\begin{align*}
(v^{(t)}_0,v^{(t)}_1):=(\tfrac{1}{N(t)}u(t,x(t)+\tfrac{x}{N(t)}),\,\tfrac{1}{N(t)^{2}}u_t(t,x(t)+\tfrac{x}{N(t)}))\in \dot{H}_x^{s_c}\times\dot{H}_x^{s_c-1}
\end{align*}
and $v^{(t)}_R$ is the solution to (NLW) with initial data $(\psi_Rv^{(t)}_0,\psi_Rv^{(t)}_1)$ given by Theorem $\ref{l22}$, then $v^{(t)}_R$ is global, satisfies the bound
\begin{align}
\lVert (v^{(t)}_R(\tau),\partial_t v^{(t)}_R(\tau))\rVert_{L_t^\infty(\mathbb{R};\dot{H}_x^{s_c}\times\dot{H}_x^{s_c-1})}< \epsilon,\label{l38}
\end{align}
and for $r\in I-t=\{s-t:s\in I\}$, and $x\in\{x\in\mathbb{R}^d:|x|\geq 2R+rN(t)\}$ we have
\begin{align}
v^{(t)}(rN(t),x)=v^{(t)}_{R}(rN(t),x)\label{l39}
\end{align}
where $v^{(t)}(\tau,x)=\tfrac{1}{N(t)}u(t+\tfrac{\tau}{N(t)},x(t)+\tfrac{x}{N(t)})$ is the solution to (NLW) with initial data $(v_0^{(t)},v_1^{(t)})$.
\end{lemma}

\begin{proof}
We argue as in \cite{KenigMerleSupercritical}.  Fix $R>0$ to be determined later in the argument and let $t\in I$ be arbitrary.  Our first goal is to obtain the global solution $v_R^{(t)}$ to (NLW) via the local well-posedness result, Theorem $\ref{l22}$.  

We begin by showing that there exists a constant $A>0$ (independent of $R$ and $t$) such that 
\begin{align}
\lVert (\psi_Rv_0^{(t)},\psi_Rv_1^{(t)})\rVert_{\dot{H}_x^{s_c}\times\dot{H}_x^{s_c-1}}\leq A.\label{l40}
\end{align}
Using Lemma $\ref{l19}$ followed by the Sobolev embedding and Remark $\ref{l9}$, we argue as follows:
\begin{align}
\nonumber &\lVert(\psi_Rv^{(t)}_0,\psi_Rv^{(t)}_1)\rVert_{\dot{H}_x^{s_c}\times\dot{H}_x^{s_c-1}}\\
\nonumber &\hspace{0.2in}\leq \lVert (\psi_R-1)v^{(t)}_0\rVert_{\dot{H}_x^{s_c}}+\lVert v^{(t)}_0\rVert_{\dot{H}_x^{s_c}}+\lVert (\psi_R-1)v^{(t)}_1\rVert_{\dot{H}_x^{s_c-1}}+\lVert v^{(t)}_1\rVert_{\dot{H}_x^{s_c-1}}\\
\nonumber &\hspace{0.2in}\lesssim \lVert |\nabla|^{s_c} (\psi_R-1)\rVert_{L_x^\frac{2d}{d-2}}\lVert v^{(t)}_0\rVert_{L_x^d}+\lVert \psi_R-1\rVert_{L_x^\infty}\lVert |\nabla|^{s_c}v^{(t)}_0\rVert_{L_x^2}+\lVert v^{(t)}_0\rVert_{\dot{H}_x^{s_c}}\\
\nonumber &\hspace{0.4in}+\lVert |\nabla|^{s_c-1} (\psi_R-1)\rVert_{L_x^\frac{2d}{d-4}}\lVert v^{(t)}_1\rVert_{L_x^\frac{d}{2}}+\lVert \psi_R-1\rVert_{L_x^\infty}\lVert |\nabla|^{s_c-1}v^{(t)}_1\rVert_{L_x^2}\\
\nonumber &\hspace{0.4in}+\lVert v^{(t)}_1\rVert_{\dot{H}_x^{s_c-1}}\\
\nonumber &\hspace{0.2in}\lesssim \bigg[\lVert \frac{1}{R^{s_c}}|\nabla|^{s_c}(\psi-1)(\frac{x}{R})\rVert_{L_x^\frac{2d}{d-2}}+\lVert \psi_R-1\rVert_{L_x^\infty}+1\bigg]\lVert v^{(t)}_0\rVert_{\dot{H}_x^{s_c}}\\
\nonumber &\hspace{0.4in}+\bigg[\lVert \frac{1}{R^{s_c-1}}|\nabla|^{s_c-1}(\psi-1)(\frac{x}{R})\rVert_{L_x^\frac{2d}{d-4}}+\lVert \psi_R-1\rVert_{L_x^\infty}+1\bigg]\lVert v^{(t)}_1\rVert_{\dot{H}_x^{s_c-1}}\\
\nonumber &\hspace{0.2in}=\bigg[\lVert |\nabla|^{s_c}(\psi-1)\rVert_{L_x^\frac{2d}{d-2}}+\lVert\psi_R-1\rVert_{L_x^\infty}+1\bigg]\lVert v^{(t)}_0\rVert_{\dot{H}_x^{s_c}}\\
\nonumber &\hspace{0.4in}+\bigg[\lVert |\nabla|^{s_c-1}(\psi-1)\rVert_{L_x^\frac{2d}{d-4}}+\lVert \psi_R-1\rVert_{L_x^\infty}+1\bigg]\lVert v^{(t)}_1\rVert_{\dot{H}_x^{s_c-1}}\\
\label{l41} &\hspace{0.2in}\lesssim \lVert (v^{(t)}_0,v^{(t)}_1)\rVert_{\dot{H}_x^{s_c}\times\dot{H}_x^{s_c-1}}.
\end{align}
where in the last inequality we note that by Remark $\ref{l8}$, $\psi-1\in C_0^\infty$ gives the finiteness of $\lVert |\nabla|^{s_c}(\psi-1)\rVert_{L_x^\frac{2d}{d-2}}$ and $\lVert |\nabla|^{s_c-1}(\psi-1)\rVert_{L_x^\frac{2d}{d-4}}$, with $s_c=2$ for $d=6$ and $\frac{2d}{d-2},\frac{2d}{d-4}\in [2,d)$ for $d\geq 7$.

Hence, by the scaling invariance ($\ref{l1}$),
\begin{align*}
\lVert (\psi_Rv^{(t)}_0,\psi_Rv^{(t)}_1)\rVert_{\dot{H}_x^{s_c}\times\dot{H}_x^{s_c-1}}&\lesssim \lVert(v_0^{(t)},v_1^{(t)})\rVert_{\dot{H}_x^{s_c}\times\dot{H}_x^{s_c-1}}\\
&\lesssim \lVert (u(t),u_t(t))\rVert_{\dot{H}_x^{s_c}\times\dot{H}_x^{s_c-1}}\\
&\lesssim \lVert (u,u_t)\rVert_{L_t^\infty(I;\dot{H}_x^{s_c}\times\dot{H}_x^{s_c-1})},
\end{align*}
and we set $A=C\lVert (u,u_t)\rVert_{L_t^\infty(I;\dot{H}_x^{s_c}\times\dot{H}_x^{s_c-1})}$ to get the desired bound.

Let us now choose $\delta_0>0$ as in Theorem $\ref{l22}$.  We next show that for every $0<\delta<\delta_0$ we may choose $R$ independent of $t$ such that 
\begin{align}
\lVert \mathcal{W}(\tau)(\psi_Rv_0^{(t)},\psi_Rv_1^{(t)})\rVert_{L_{\tau,x}^{d+1}}<\delta.\label{l42}
\end{align}
To do so, using the Strichartz inequality we see that it suffices to prove
\begin{align}
\lVert (\psi_Rv_0^{(t)},\psi_Rv_1^{(t)})\rVert_{\dot{H}_x^{s_c}\times\dot{H}_x^{s_c-1}}<\frac{\delta}{C},\label{l43}
\end{align}
where $C$ is the constant from the Strichartz inequality.  Suppose for contradiction that the claim ($\ref{l43}$) failed.  We may then choose $\delta'_0>0$ together with sequences $R_n\rightarrow\infty$ and $t_n\in I$ such that for each $n\in\mathbb{N}$
\begin{align}
\lVert (\psi_{R_n}v_{0}^{(t_n)},\psi_{R_n}v_{1}^{(t_n)})\rVert_{\dot{H}_x^{s_c}\times\dot{H}_x^{s_c-1}}>\delta'_0,\label{l44}
\end{align}
where $(v_{0}^{(t_n)},v_1^{(t_n)})$ is the pair defined in the statement of the theorem.  Since $u$ is almost periodic, we may then choose $(f,g)\in\dot{H}_x^{s_c}\times\dot{H}_x^{s_c-1}$ such that $(v_0^{(t_n)},v_1^{(t_n)})$ converges to $(f,g)$ in $\dot{H}_x^{s_c}\times\dot{H}_x^{s_c-1}$.  Moreover, the density of $C_0^\infty\times C_0^\infty$ in $\dot{H}_x^{s_c}\times\dot{H}_x^{s_c-1}$ allows us to to choose $(f_m,g_m)\in C_0^\infty\times C_0^\infty(\mathbb{R}^d)$ with $(f_m,g_m)$ converging to $(f,g)$ in $\dot{H}_x^{s_c}\times\dot{H}_x^{s_c-1}$.

Thus, invoking ($\ref{l41}$) and using Minkowski's inequality, we obtain
\begin{align}
\nonumber &\lVert (\psi_{R_n}v_0^{(t_n)},\psi_{R_n}v_1^{(t_n)})\rVert_{\dot{H}_x^{s_c}\times\dot{H}_x^{s_c-1}}\\
\nonumber &\hspace{0.2in} \lesssim\lVert (\psi_{R_n}(v_{0}^{(t_n)}-f),\psi_{R_n}(v_1^{(t_n)}-g))\rVert_{\dot{H}_x^{s_c}\times\dot{H}_x^{s_c-1}}\\
\nonumber &\hspace{0.4in} +\lVert (\psi_{R_n}(f-f_m),\psi_{R_n}(g-g_m))\rVert_{\dot{H}_x^{s_c}\times\dot{H}_x^{s_c-1}}\\
\nonumber &\hspace{0.4in} +\lVert (\psi_{R_n}f_m,\psi_{R_n}g_m)\rVert_{\dot{H}_x^{s_c}\times\dot{H}_x^{s_c-1}}\\
\nonumber &\hspace{0.2in}\lesssim \lVert (v^{(t_n)}_{0}-f,v^{(t_n)}_{1}-g)\rVert_{\dot{H}_x^{s_c}\times\dot{H}_x^{s_c-1}}+\lVert (f-f_m,g-g_m)\rVert_{\dot{H}_x^{s_c}\times\dot{H}_x^{s_c-1}}\\
\label{l45}&\hspace{0.4in}+\lVert (\psi_{R_n}f_m,\psi_{R_n}g_m)\rVert_{\dot{H}_x^{s_c}\times\dot{H}_x^{s_c-1}}.
\end{align}
where we note that ($\ref{l41}$) holds for any $(v_0,v_1)\in\dot{H}_x^{s_c}\times\dot{H}_x^{s_c-1}$.  As $(f_m,g_m)\in C_0^\infty\times C_0^\infty$ and $\supp \psi_{R_n}\subset \{x:|x|>R_n\}$, we have 
\begin{align*}
\psi_{R_n}f_m\equiv \psi_{R_n}g_m\equiv 0.
\end{align*}
for $n$ sufficiently large.  Thus, taking the limit $n\rightarrow \infty$ in ($\ref{l45}$) followed by the limit $m\rightarrow\infty$ yields
\begin{align*}
\lVert (\psi_{R_n}v_{0}^{(t_n)},\psi_{R_n}v_{1}^{(t_n)})\rVert_{\dot{H}_x^{s_c}\times\dot{H}_x^{s_c-1}}\mathop{\longrightarrow}_{n\rightarrow\infty} 0.
\end{align*}
But this contradicts $(\ref{l44})$, proving that the desired estimate ($\ref{l43}$) holds.

Collecting ($\ref{l40}$) and ($\ref{l42}$), Theorem $\ref{l22}$ now implies that there exists a global solution $v_R^{(t)}$ with the bounds
\begin{align}
\label{l46}\lVert v^{(t)}_R\rVert_{L_{t,x}^{d+1}}\lesssim \epsilon_1,
\end{align}
\begin{align}
\label{l47}\lVert |\nabla|^{\frac{d^2-4d+1}{2(d-1)}}v^{(t)}_R\rVert_{S_{\frac{d+1}{2(d-1)}}(\mathbb{R})}+\lVert |\nabla|^{\frac{d^2-4d+1}{2(d-1)}-1}\partial_t v^{(t)}_R\rVert_{S_{\frac{d+1}{2(d-1)}}(\mathbb{R})}<\infty.
\end{align}
Moreover, using the Stricharz inequality and Lemma $\ref{l21}$ followed by the bounds ($\ref{l44}$), ($\ref{l46}$) and ($\ref{l47}$), we obtain 
\begin{align*}
&\lVert (v^{(t)}_R,\partial_t v^{(t)}_R)\rVert_{L_t^\infty(\mathbb{R};\dot{H}_x^{s_c}\times\dot{H}_x^{s_c-1})}\\
&\hspace{0.3in}\lesssim \lVert (\psi_Rv^{(t)}_0,\psi_Rv^{(t)}_1)\rVert_{\dot{H}_x^{s_c}\times\dot{H}_x^{s_c-1}}+\lVert |\nabla|^{\frac{d^2-4d+1}{2(d-1)}}(v^{(t)}_R)^3\rVert_{N_{\frac{d-3}{2(d-1)}}(\mathbb{R})}\\
&\hspace{0.3in}\lesssim \delta+\lVert |\nabla|^{\frac{d^2-4d+1}{2(d-1)}}v^{(t)}_R\rVert_{S_{\frac{d+1}{2(d-1)}}(\mathbb{R})}\lVert v^{(t)}_R\rVert_{L_{t,x}^{d+1}}^2\\
&\hspace{0.3in}\lesssim \delta+\delta^2.
\end{align*} 
Thus, choosing $\delta$ small enough such that $C(\delta+\delta^2)<\epsilon$ gives the bound $(\ref{l38})$ as desired.

Finally, we now address $(\ref{l39})$.  Given $t\in I$ and $r\in I-t\cap [0,\infty)$ we note that
\begin{align*}
v^{(t)}_{R}(0,x)&=v^{(t)}(0,x)\quad\textrm{and}\quad \partial_t v^{(t)}_R(0,x)=\partial_t v^{(t)}(0,x)
\end{align*}
on $|x|>2R$.  Then, the finite speed of propagation implies 
\begin{align*}
v^{(t)}_R(rN(t),x)=v^{(t)}(rN(t),x)
\end{align*}
on $|x|>2R+rN(t)$ as desired.
\end{proof}

\section{Finite time blow-up solution}
In this section, we show that the finite time blow-up solution described in Theorem $\ref{l16}$ cannot exist.  Arguing as in \cite{KenigMerleSupercritical,KillipVisanSupercriticalNLW3D}, we prove that the solution must have zero energy, contradicting the fact that the solution blows up.  We note that without loss of generality we may assume $\sup I=1$.

The first step is to note that the function $N(t)$ tends to infinity as $t$ approaches the blow-up time.  In the context of the nonlinear Schr\"odinger equation this property is given in \cite{KillipTaoVisanCubic,KillipVisanClayNotes}, while for the nonlinear wave equation, see \cite{KenigMerleSupercritical,KillipVisanSupercriticalNLW3D}.

\begin{lemma}
\label{l48}
Let $u:I\times\mathbb{R}^d\rightarrow\mathbb{R}$ be an almost periodic solution to (NLW) with maximal interval of existence $I$, $sup I=1$.  Then there exist $\epsilon>0$ and $C>0$ such that for all $t\in (1-\epsilon,1)$,
\begin{align*}
N(t)\geq \frac{C}{1-t}.
\end{align*}
\end{lemma}
\begin{proof}
Suppose for contradiction that the claim failed, and let us choose a sequence $t_n\rightarrow 1$ such that for all $n\in\mathbb{N}$, $N(t_n)(1-t_n)<\frac{1}{n}$.  For all $n\in\mathbb{N}$, we set
\begin{align*}
(v_{0,n},v_{1,n})=(\tfrac{1}{N(t_n)}u(t_n,x(t_n)+\tfrac{x}{N(t_n)}),\tfrac{1}{N(t_n)^2}u_t(t_n,x(t_n)+\tfrac{x}{N(t_n)}))
\end{align*}
and let $v_n$ denote the solution to (NLW) with Cauchy data $(v_{0,n},v_{1,n})$, with maximal interval of existence $I_n$.  Then for all $n\in\mathbb{N}$, the scaling and space translation symmetries imply that we have $\sup I_n=N(t_n)(1-t_n)$.

Note that since $u$ is almost periodic, we may choose $(f,g)\in \dot{H}_x^{s_c}\times \dot{H}_x^{s_c-1}$ such that $(v_{0,n},v_{1,n})\rightarrow (f,g)$ in $\dot{H}_x^{s_c}\times \dot{H}_x^{s_c-1}$ as $n\rightarrow\infty$.  

Let $\delta_0(d,\lVert (u,u_t)\rVert_{L_t^\infty(I;\dot{H}_x^{s_c}\times\dot{H}_x^{s_c-1})})>0$ as in Theorem $\ref{l22}$.  Then there exists an open interval $0\in J\subset\mathbb{R}$ small enough so that
\begin{align*}
\lVert \mathcal{W}(t)(f,g)\rVert_{L_{t,x}^{d+1}(J\times\mathbb{R}^d)}<\frac{\delta_0}{3}.
\end{align*}

On the other hand the Strichartz inequality gives
\begin{align*}
\lVert \mathcal{W}(t)(f,g)-\mathcal{W}(t)(v_{0,n},v_{1,n})\rVert_{L_{t,x}^{d+1}(J\times\mathbb{R}^d)}\rightarrow 0
\end{align*}
as $n\rightarrow\infty$, so that we may choose $N$ large enough such that for every $n\geq N$, $\lVert \mathcal{W}(t)(v_{0,n},v_{1,n})\rVert_{L_{t,x}^{d+1}(J\times\mathbb{R}^d)}\leq \frac{2\delta_0}{3}$.  Thus for all $n\geq N$, Theorem $\ref{l22}$ implies that $J\subset I_n$, and thus $\frac{1}{2}\sup J\in I_n$.  However, this contradicts the limit $\sup I_n\rightarrow 0$ as $n\rightarrow\infty$.  Thus, the desired claim holds.
\end{proof}

A second ingredient that is necessary to rule out the finite time blow-up solution is to control its support.
\begin{lemma}
\label{l49}
Let $u:I\times\mathbb{R}^d\rightarrow\mathbb{R}$ be an almost periodic solution to (NLW) with maximal interval of existence $I$, $\sup I=1$ and $(u,u_t)\in L_t^\infty(I;\dot{H}_x^{s_c}\times\dot{H}_x^{s_c-1})$.

Then there exists $y\in\mathbb{R}^d$ such that for each $0<s<1$, we have
\begin{align*}
\supp u(s,\cdot),\quad \supp u_t(s,\cdot)\subset \overline{B(y,1-s)}
\end{align*}
\end{lemma}
\begin{proof}
We argue as in \cite{KenigMerleNLW,KenigMerleSupercritical}.  Fix $\epsilon>0$ and $0<s<1$.  Let $R,v^{(t)}_0,v^{(t)}_1,v^{(t)}_R$ be as stated in Lemma $\ref{l37}$.  

We first show 
\begin{align}
\limsup_{t\rightarrow 1} \int_{|x-x(t)|\geq \frac{2R}{N(t)}+t-s} |\nabla u(s,x)|^\frac{d}{2}+|u_t(s,x)|^\frac{d}{2}dx\leq C\epsilon.
\label{l50}
\end{align}
Indeed, for $t\in I$,
\begin{align*}
&\int_{|x-x(t)|\geq \frac{2R}{N(t)}+t-s} |(\nabla u)(s,x)|^\frac{d}{2}dx\\
&\hspace{0.2in}=\int_{|x|\geq 2R+(t-s)N(t)} \left|(\nabla u)\left(s,x(t)+\tfrac{x}{N(t)}\right)\right|^\frac{d}{2}\tfrac{1}{N(t)^{d}}dx\\
&\hspace{0.2in}\leq \int_{\mathbb{R}^d} |\nabla v^{(t)}_{R}((s-t)N(t),x)|^\frac{d}{2}dx\\
&\hspace{0.2in}\lesssim \lVert v^{(t)}_R((s-t)N(t),x)\rVert_{\dot{H}_x^{s_c}}^\frac{d}{2}\\
&\hspace{0.2in}\lesssim \epsilon
\end{align*}
where to obtain the last two inequalities, we used Sobolev's inequality combined with Lemma $\ref{l37}$.  A similar argument also shows the corresponding inequality with $\nabla u(s,x)$ replaced by $u_s(s,x)$.  As $t\in I$ is arbitrary, this proves the desired inequality ($\ref{l50}$).

We next show that there exists $\epsilon'>0$ and $A>0$ such that for all $1-\epsilon'<t<1$, we have
\begin{align}
|x(t)|<A.\label{l51}
\end{align}
To see this, suppose for a contradiction that the claim failed.  Then there exists a sequence of times $\{t_n\}$ such that $t_n\in (1-\frac{1}{n},1)$ and $|x(t_n)|>n$ for all $n\in\mathbb{N}$.  Then given $M>0$, $|x|<M$ implies $|x-x(t_n)|\geq n-M$.  Moreover, by Lemma $\ref{l48}$, $N(t_n)\rightarrow \infty$ as $t_n\rightarrow 1$ which yields $\tfrac{2R}{N(t_n)}\rightarrow0$ as $n\rightarrow\infty$, so that for $n$ large enough, $\tfrac{2R}{N(t_n)}\leq 1$.  Noting that for all $n\in\mathbb{N}$, $t_n\leq 1$, we deduce that for $n$ large enough,
\begin{align*}
\{x:|x|<M\}\subset \{x:|x-x(t_n)|\geq \tfrac{2R}{N(t_n)}+t_n\}.
\end{align*}
Using this embedding to expand the domain of integration in ($\ref{l50}$), we obtain 
\begin{align*}
\int_{|x|<M} |\nabla u(0,x)|^\frac{d}{2}+|u_t(0,x)|^\frac{d}{2}dx&\leq 2C\epsilon.
\end{align*}
Letting $\epsilon\rightarrow 0$ followed by $M\rightarrow\infty$, we derive $\int_{\mathbb{R}^d} |\nabla u(0,x)|^\frac{d}{2}+|u_t(0,x)|^\frac{d}{2}dx=0$, and hence $u\equiv 0$.  This contradicts the fact that $u$ is a blow-up solution, and thus the desired claim ($\ref{l51}$) holds.

With the bound ($\ref{l51}$) in hand, we are now ready to conclude the proof of the lemma.  Let us choose a time sequence $t_n\in (1-\epsilon',1)$ such that $t_n\rightarrow 1$ as $n\rightarrow\infty$.  Then by ($\ref{l51}$), $|x(t_n)|<A$ for all $n$, so that we may choose a subsequence (still labeled $t_n$) such that $x(t_n)\rightarrow y$ as $n\rightarrow\infty$.

We now claim that for $\eta>0$ fixed and for $n$ large enough (depending on $\eta$),
\begin{align}
\{x:|x-y|\geq 1-s+\eta\}\subset \{x:|x-x(t_n)|\geq \tfrac{2R}{N(t_n)}+t_n-s\}.
\label{l52}
\end{align}
To observe this inclusion, by the convergence of $x(t_n)$ let us choose $N_0\in\mathbb{N}$ such that for all $n>N_0$, $|x(t_n)-y|<\frac{\eta}{2}$.  Then for $n>N_0$ and $|x-y|\geq 1-s+\eta$, we have
\begin{align}
|x-x(t_n)|&\geq 1-s+\tfrac{\eta}{2}.
\label{l53}
\end{align}
Moreover, by Lemma $\ref{l48}$ $N(t_n)\rightarrow \infty$ as $t_n\rightarrow 1$, so that we may choose $N_1\in\mathbb{N}$ such that for all $n>N_1$, 
\begin{align}
\label{l54}
\tfrac{2R}{N(t_n)}<\tfrac{\eta}{2}.
\end{align}
Putting together $(\ref{l53})$ and $(\ref{l54})$ and recalling $t_n<1$, we obtain that for $n>\max\{N_0,N_1\}$,
\begin{align*}
|x-x(t_n)|&\geq \tfrac{2R}{N(t_n)}+t_n-s.
\end{align*}

Returning back to ($\ref{l50}$) and invoking ($\ref{l52}$) followed by letting $n\rightarrow\infty$, we get 
\begin{align}
\int_{|x-y|\geq 1-s+\eta} |\nabla u(s,x)|^\frac{d}{2}+|u_t(s,x)|^\frac{d}{2}dx\leq C\epsilon.\label{l55}
\end{align}
Letting $\eta\rightarrow 0$ and using the monotone convergence theorem together with $\epsilon\rightarrow 0$, we deduce
\begin{align*}
\int_{|x-y|\geq 1-s} |\nabla u(s,x)|^\frac{d}{2}+|u_t(s,x)|^\frac{d}{2}dx=0.
\end{align*}
This immediately implies $\supp u_t(s)\subset \overline{B(y,1-s)}$.

To conclude, we note that $(\ref{l55})$ also implies that $u(s)$ is constant on $\{|x-y|> 1-s\}$.  Then $u\in L_t^\infty \dot{H}_x^{s_c}$ gives $u\in L_t^\infty L_x^{d}$ via the Sobolev embedding.  This in turn forces $u=0$ on $\{|x-y|> 1-s\}$, and thus $\supp u\subset \overline{B(y,1-s)}$ as desired.
\end{proof}

Arguing as in \cite{KenigMerleSupercritical}, we can now rule out the finite time blow-up solution:
\begin{proposition}
There is no solution $u:I\times \mathbb{R}^d\rightarrow\mathbb{R}$ to (NLW) with maximal interval of existence $I$ satisfying the properties of a finite time blow-up solution in the sense of Theorem \ref{l16}.
\end{proposition}

\begin{proof}
Let us suppose for a contradiction that there is such a solution $u$.  By the time-reversal and scaling symmetries we may assume that $\sup I=1$.  Using Lemma $\ref{l49}$ and the space-translation symmetry, we may further assume that $\supp u(t),\supp u_t(t)\subset \overline{B(0,1-t)}$.  Then for all $t\in (0,1)$, we have
\begin{align*}
E(u(t),u_t(t))&=\int_{|x|\leq 1-t} \tfrac{1}{2}|\nabla u(t)|^2+\tfrac{1}{2}|u_t(t)|^2+\tfrac{1}{4}|u(t)|^{4}dx\\
&\lesssim (1-t)^{d-4}[\lVert \nabla u(t)\rVert_{L_x^\frac{d}{2}(\mathbb{R}^d)}^2+\lVert u_t(t)\rVert_{L_x^\frac{d}{2}(\mathbb{R}^d)}^2+\lVert u(t)\rVert_{L_x^{d}(\mathbb{R}^d)}^{4}]\\
&\lesssim (1-t)^{d-4}[\lVert u(t)\rVert_{\dot{H}_x^{s_c}}^2+\lVert u_t(t)\rVert_{\dot{H}_x^{s_c-1}}^2+\lVert u(t)\rVert_{\dot{H}_x^{s_c}}^4]\\
&\lesssim (1-t)^{d-4}
\end{align*}
where we have used the fact that $u\in L_t^\infty(I;\dot{H}_x^{s_c}\times \dot{H}_x^{s_c-1})$.

Letting $t\nearrow 1$ and using the conservation of energy,
\begin{align*}
E(u(0),u_t(0))=\lim_{t\rightarrow 1}E(u(t),u_t(t))=0.
\end{align*}
This implies $u\equiv 0$ which contradicts the assumption that $u$ is a finite time blow-up solution.  Thus such a solution cannot exist. 
\end{proof} 

\section{Additional decay}

In this section, we prove that the soliton-like and frequency cascade solutions identified in Theorem $\ref{l15}$ satisfy an additional decay property.  More precisely, for $d\geq 6$ we show that $(u,u_t)\in L_t^\infty(\mathbb{R};\dot{H}_x^{1-\epsilon}\times \dot{H}_x^{-\epsilon})$ for some $\epsilon=\epsilon(d)>0$.  In particular, we obtain that such solutions belong to $L_t^\infty(\dot{H}_x^1\times L_x^2)$.  Our approach follows that of Killip and Visan in \cite{KillipVisanECritical,KillipVisanSupercriticalNLS,KillipVisanSupercriticalNLW3D}.  

The main result of this section is the following:
\begin{theorem}
\label{l56}
Assume $d\geq 6$ and that $u:\mathbb{R}\times\mathbb{R}^d\rightarrow\mathbb{R}$ is an almost periodic solution to (NLW) with $(u,u_t)\in L_t^\infty(\mathbb{R};\dot{H}_x^{s_c}\times \dot{H}_x^{s_c-1})$ and 
\begin{align*}
\inf_{t\in I} N(t)\geq 1.
\end{align*}

Then we have 
\begin{align}\label{l57}
(u,u_t)\in L_t^\infty(\mathbb{R};\dot{H}_x^{1-\epsilon}\times \dot{H}_x^{-\epsilon}) 
\end{align}
for some $\epsilon=\epsilon(d)>0$.  In particular, $(u,u_t)\in L_t^\infty(\mathbb{R};\dot{H}_x^1\times L_x^2)$.
\end{theorem}

Arguing as in \cite{KillipVisanECritical,KillipVisanSupercriticalNLS,KillipVisanSupercriticalNLW3D}, we obtain Theorem $\ref{l56}$ in two steps.  The first step is to prove that the solution $u$ belongs to $L_t^\infty L_x^{q_0}$ for all $q_0\in (\frac{2(d-1)}{d-3},d]$.  The second step is to perform a double Duhamel technique \cite{CKSTT,Tao} to improve this decay to $(u,u_t)\in L_t^\infty(\dot{H}_x^{s_c-s_0}\times \dot{H}_x^{s_c-1-s_0})$ for some $s_0=s_0(d,q_0)>0$.  Iterating the second step finitely many times, we obtain Theorem $\ref{l56}$.  

More precisely, Theorem $\ref{l56}$ will follow once we establish the following two lemmas:
\begin{lemma}
\label{l58}
Suppose $d\geq 6$ and that $u:\mathbb{R}\times\mathbb{R}^d\rightarrow\mathbb{R}$ is an almost periodic solution to (NLW) with $(u,u_t)\in L_t^{\infty}(\mathbb{R};\dot{H}_x^{s_c}\times \dot{H}_x^{s_c-1})$ and 
\begin{align}
\label{l59}\inf_{t\in I} N(t)\geq 1. 
\end{align}

Then for every $q_0\in (\frac{2(d-1)}{d-3},d]$ we have $u\in L_t^\infty L_x^{q_0}$.
\end{lemma}

\begin{lemma}
\label{l60}
Suppose $d\geq 6$ and that $u:\mathbb{R}\times\mathbb{R}^d\rightarrow\mathbb{R}$ is an almost periodic solution to (NLW) with $(u,u_t)\in L_t^\infty(\mathbb{R};\dot{H}_x^{s_c}\times\dot{H}_x^{s_c-1})$ and 
\begin{align*}
\inf_{t\in I} N(t)\geq 1.
\end{align*}  

Moreover, assume that there exists $4<q_1<d-1$ and $s\in [1,s_c]$ such that $u\in L_t^\infty L_x^{q_1}$ and $|\nabla|^s u\in L_t^\infty L_x^2$.  Then 
\begin{align}
\label{l61}(u,u_t)\in L_t^\infty(\mathbb{R}; \dot{H}_x^{s-s_0}\times \dot{H}_x^{s-1-s_0}).
\end{align}
for some $s_0=s_0(d,q_1)>0$.
\end{lemma}

We will discuss the proofs of Lemma $\ref{l58}$ and Lemma $\ref{l60}$ in detail in the rest of this section; however, with these two lemmas in hand, we immediately complete the proof of the main theorem of this section.
\begin{proof}[Proof of Theorem \ref{l56}.]
We begin by choosing a suitable exponent to be able to apply Lemma $\ref{l58}$ and Lemma $\ref{l60}$.  To this end, we define 
\begin{align*}
q(d)&:=\tfrac{d^2-d-2}{2(d-3)}
\end{align*}
and note that $d\geq 6$ implies $q(d)\in (\frac{2(d-1)}{d-3},d)$ and $4<q(d)<d-1$.

Fix $s_0=s_0(d,q(d))$ as in Lemma $\ref{l60}$.  By induction, we now prove that for each $k\in\mathbb{N}$ with $s_c-(k-1)s_0\geq 1$, we have $(u,u_t)\in L_t^\infty(\mathbb{R};\dot{H}_x^{s_c-ks_0}\times\dot{H}_x^{s_c-1-ks_0})$.  We first note that for $k=0$ the result follows from the hypothesis $(u,u_t)\in L_t^\infty (\mathbb{R};\dot{H}_x^{s_c}\times\dot{H}_x^{s_c-1})$.  For the induction step, we assume that the result holds for some $k-1\in\mathbb{N}$ with $s_c-(k-2)s_0\geq 1$.  We then have $u\in L_t^\infty \dot{H}_x^{s_c-(k-1)s_0}$, so that if $k$ also satisfies $s_c-(k-1)s_0\geq 1$, then an immediate application of Lemma $\ref{l60}$ gives 
\begin{align*}
(u,u_t)\in L_t^\infty(\mathbb{R};\dot{H}_x^{s_c-ks_0}\times \dot{H}_x^{s_c-1-ks_0}) 
\end{align*}
establishing the induction step.

Note that taking $k\in\mathbb{N}$ as the largest integer such that $s_c-(k-1)s_0\geq 1$ we obtain the desired result ($\ref{l57}$) with $\epsilon=1-(s_c-ks_0)$.
\end{proof}

We now turn our attention to the proofs of Lemma $\ref{l58}$ and Lemma $\ref{l60}$.  The rest of this section is devoted to proving these two lemmas.  We start with,
\subsection{Proof of Lemma \ref{l58}.}\ \\
\hspace*{0.1in} Let $\eta>0$ be a small constant to be chosen later.  Assume $u$ is a solution to (NLW) as stated in Lemma $\ref{l58}$.  Then almost periodicity together with the condition ($\ref{l59}$) imply that we may find a dyadic number $N_0$ such that
\begin{align}
\lVert |\nabla|^{s_c}u_{\leq N_0}\rVert_{L_t^\infty L_x^2}\leq \eta.
\label{l62}
\end{align}

Let us now fix $R\in (\frac{2(d-1)}{d-3},\min\{\frac{2d}{d-4},\frac{3d}{d-1}\})$ and define
\begin{align*}
\mathcal{S}(N)=N^{\frac{d}{R}-1}\lVert u_N\rVert_{L_t^\infty L_x^R}
\end{align*}
for each dyadic number $N\in \{2^n:n\in\mathbb{Z}\}$.  

To prove Lemma $\ref{l58}$, it is enough to show $\lVert u_N\rVert_{L_t^\infty L_x^R}\lesssim N^\gamma$ for some $\gamma>0$ and $N$ sufficiently small depending on $u$, $d$ and $R$ (see the argument at the end of this section).  This bound will follow from the following decay estimate, which uses a Gronwall type inequality as stated in \cite{KillipVisanSupercriticalNLW3D}.
\begin{lemma}[Decay estimate]
\label{l63}
For all dyadic numbers $N\leq 8N_0$, we have
\begin{align}
\label{l64}
\nonumber \mathcal{S}(N)&\lesssim (\tfrac{N}{N_0})^{d-\frac{d}{R}-3}
+\eta\sum_{N_1=\frac{2N}{8}}^{N_0}\left[\left(\tfrac{N}{N_1}\right)^{d-\frac{d}{R}-3}\mathcal{S}(N_1)\right]\\
&\hspace{1.8in}+\eta\sum_{N_1\leq \frac{N}{8}} \left[\left(\tfrac{N_1}{N}\right)^{\frac{d}{R}-\frac{d}{2}+2}\mathcal{S}(N_1)\right].
\end{align}
In particular, 
\begin{align}
\label{l65}\mathcal{S}(N)&\lesssim N^{\frac{d-4}{2}}
\end{align}
for every $N\leq 8N_0$
\end{lemma}

\begin{proof}
We argue as in \cite{KillipVisanECritical,KillipVisanSupercriticalNLS}.  Let $N\leq 8N_0$.  We first observe that by Bernstein's inequality together with the Sobolev embedding and $u\in L_t^\infty \dot{H}_x^{s_c}$,
\begin{align*}
\mathcal{S}(N)&\lesssim N^{\frac{d}{2}-1}\lVert u_N\rVert_{L_t^\infty L_x^2}\lesssim \lVert |\nabla|^{s_c}u_N\rVert_{L_t^\infty L_x^2}<\infty,
\end{align*}

We now turn our attention to ($\ref{l64}$).  We first note that using the time translation symmetry, it suffices to prove the result when $t=0$.  Then, by using the Duhamel formula ($\ref{l18}$) combined with Minkowski's inequality, we obtain
\begin{align}
&\nonumber N^{\frac{d}{R}-1}\lVert u_N(0)\rVert_{L_x^R}\\
\nonumber &\hspace{0.2in}\lesssim N^{\frac{d}{R}-1}\bigg(\int_0^{N^{-1}} \lVert \tfrac{\sin(-t'|\nabla|)}{|\nabla|}P_NF(u(t'))\rVert_{L_x^R}dt'\\
&\hspace{1.4in}+\int_{N^{-1}}^\infty \lVert \tfrac{\sin(-t'|\nabla|)}{|\nabla|}P_NF(u(t'))\rVert_{L_x^R}dt'\bigg).
\label{l66}
\end{align}
We then use Bernstein's inequality on the first term and the dispersive inequality ($\ref{l5}$) on the second term to obtain
\begin{align}
\nonumber (\ref{l66})&\lesssim N^{\frac{d}{R}-1}\bigg(\int_0^{N^{-1}} N^{\frac{d}{2}-\frac{d}{R}}\lVert \tfrac{\sin(-t'|\nabla|)}{|\nabla|}P_NF(u(t'))\rVert_{L_x^2}dt'\\
\nonumber &\hspace{0.85in}+\int_{N^{-1}}^\infty |t'|^{-(d-1)(\frac{1}{2}-\frac{1}{R})}\lVert |\nabla|^{\frac{d-1}{2}-\frac{d+1}{R}}P_NF(u(t'))\rVert_{L_x^{R'}}dt'\bigg)\\
\nonumber &\lesssim N^{\frac{d}{R}-1}\bigg(\int_0^{N^{-1}} N^{\frac{d}{2}-\frac{d}{R}}\lVert |\nabla|^{-1}P_NF(u(t'))\rVert_{L_x^2}dt'\\
\nonumber &\hspace{0.85in}+\int_{N^{-1}}^\infty |t'|^{-(d-1)(\frac{1}{2}-\frac{1}{R})}\lVert |\nabla|^{\frac{d-1}{2}-\frac{d+1}{R}}P_NF(u(t'))\rVert_{L_x^{R'}}dt'\bigg)\\
\nonumber &\lesssim N^{\frac{d}{2}-3}\lVert P_NF(u)\rVert_{L_t^\infty L_x^2}+N^{d-\frac{d}{R}-3}\lVert P_NF(u)\rVert_{L_t^\infty L_x^{R'}}\\
&\lesssim N^{d-\frac{d}{R}-3}\lVert P_NF(u)\rVert_{L_t^\infty L_x^{R'}}\label{l67}
\end{align}
where in passing from the the first line to the third we use ($\ref{l5}$) once more and in passing from the fourth line to the fifth line, we used the fact that $(d-1)(\frac{1}{2}-\frac{1}{R})>1$ to observe the finiteness of the integral.

Collecting ($\ref{l66}$) and ($\ref{l67}$), we obtain 
\begin{align*}
N^{\frac{d}{R}-1}\lVert u_N(0)\rVert_{L_x^R}\lesssim N^{d-\frac{d}{R}-3}\lVert P_NF(u)\rVert_{L_t^\infty L_x^{R'}}.
\end{align*}

Now to establish ($\ref{l64}$), it remains to estimate the term $\lVert P_NF(u)\rVert_{L_t^\infty L_x^{R'}}$.  We start by decomposing $u$ as
\begin{align*}
u&=u_{\leq \frac{N}{8}}+u_{\frac{N}{8}<\cdot\leq N_0}+u_{>N_0}\\
&=:u_1+u_2+u_3.
\end{align*}
Note that this decomposition gives
\begin{align*}
\lVert P_N(u^3)\rVert_{L_t^\infty L_x^{R'}}&=\lVert P_N\bigg(\sum_{i=1}^3 u_i\bigg)^3\rVert_{L_t^\infty L_x^{R'}}\\
&=\lVert \sum_{i,j,k=1}^3 P_N(u_iu_ju_k)\rVert_{L_t^\infty L_x^{R'}}\\
&\lesssim \lVert P_N(u_1^3)\rVert_{L_t^\infty L_x^{R'}}+\lVert P_N(u_2^3)\rVert_{L_t^\infty L_x^{R'}}\\
&\hspace{0.1in}+\sum_{i,j=1}^3 \lVert P_N(u_3u_iu_j)\rVert_{L_t^\infty L_x^{R'}}+\sum_{i=1}^2 \lVert P_N(u_2u_1u_i)\rVert_{L_t^\infty L_x^{R'}},
\end{align*}
where we have grouped some terms.

Using this inequality combined with the boundedness of $P_N$, we obtain
\begin{align}
\nonumber N^{\frac{d}{R}-1}\lVert u_N(0)\rVert_{L_x^R}&\lesssim N^{d-\frac{d}{R}-3}\lVert P_NF(u)\rVert_{L_t^\infty L_x^{R'}}\\
\nonumber &\leq N^{d-\frac{d}{R}-3}\bigg(\lVert P_Nu_1^3\rVert_{L_t^\infty L_x^{R'}}+\lVert u_2^3\rVert_{L_t^\infty L_x^{R'}}\\
\nonumber &\hspace{0.2in}+\sum_{i,j=1}^3\lVert u_3u_iu_j\rVert_{L_t^\infty L_x^{R'}}+\sum_{i=1}^2 \lVert u_1u_2u_i\rVert_{L_t^\infty L_x^{R'}}\bigg)\\
&=N^{d-\frac{d}{R}-3}\bigg((I)+(II)+(III)_{i,j}+(IV)_i\bigg)\label{l68}
\end{align}

We now estimate each of the above terms $(I)$, $(II)$, $(III)_{i,j}$ and $(IV)_{i}$ separately.

\underline{\it Term (I)}: 
By the support of the Fourier transform of $u_{\leq\frac{N}{8}}(t)^3$, we have
\begin{align}
\label{l69} P_N[u_{\leq \frac{N}{8}}(t)^3]\equiv 0,
\end{align}
so that $(I)=0$.

\underline{\it Term (II)}: 
Using H\"older's inequality, the Sobolev embedding, and the boundedness of $P_{>\frac{N}{8}}$ together with Bernstein's inequality, we obtain
\begin{align}
&\nonumber \lVert u_{2}^3\rVert_{L_t^\infty L_x^{R'}}\leq \lVert u_{2}\rVert_{L_t^\infty L_x^d}\lVert u_{2}\rVert_{L_t^\infty L_x^\frac{2Rd}{Rd-d-R}}^2\\
\nonumber &\hspace{0.2in}\lesssim \lVert |\nabla|^{s_c}u_{\leq N_0}\rVert_{L_t^\infty L_x^2}\bigg[\sum_{\frac{2N}{8}\leq N_1\leq N_2\leq N_0}\lVert u_{N_1}\rVert_{L_t^\infty L_x^\frac{2Rd}{Rd-d-R}}\lVert u_{N_2}\rVert_{L_t^\infty L_x^\frac{2Rd}{Rd-d-R}}\bigg]\\
\nonumber &\hspace{0.2in}\lesssim \lVert |\nabla|^{s_c}u_{\leq N_0}\rVert_{L_t^\infty L_x^2}\\
\nonumber &\hspace{0.4in}\bigg[\sum_{\frac{2N}{8}\leq N_1\leq N_2\leq N_0}N_1^{\frac{3d-Rd+R}{2R}}\lVert u_{N_1}\rVert_{L_t^\infty L_x^R}N_2^{-\frac{Rd-d-3R}{2R}}\lVert |\nabla|^\frac{Rd-d-3R}{2R}u_{N_2}\rVert_{L_t^\infty L_x^\frac{2Rd}{Rd-d-R}}\bigg]\\ 
\nonumber &\hspace{0.2in}\lesssim \lVert |\nabla|^{s_c}u_{\leq N_0}\rVert_{L_t^\infty L_x^2}\\
\nonumber &\hspace{0.4in}\bigg[\sum_{N_1=\frac{2N}{8}}^{N_0} \bigg\{N_1^{\frac{3d-Rd+R}{2R}}\lVert u_{N_1}\rVert_{L_t^\infty L_x^R}\\
\nonumber &\hspace{1.75in}\left(\sum_{N_2=N_1}^{N_0}N_2^{-\frac{Rd-d-3R}{2R}}\lVert |\nabla|^{s_c}u_{N_2}\rVert_{L_t^\infty L_x^2}\right)\bigg\}\bigg]\\
\nonumber &\hspace{0.2in}\lesssim \lVert |\nabla|^{s_c}u_{\leq N_0}\rVert_{L_t^\infty L_x^2}\\
\nonumber &\hspace{0.4in}\bigg[\sum_{N_1=\frac{2N}{8}}^{N_0} \bigg\{N_1^{\frac{3d-Rd+R}{2R}}\lVert u_{N_1}\rVert_{L_t^\infty L_x^R} \\
\nonumber &\hspace{1.75in}\left(N_1^{-\frac{Rd-d-3R}{2R}}\lVert (u,u_t)\rVert_{L_t^\infty(\mathbb{R};\dot{H}_x^{s_c}\times\dot{H}_x^{s_c-1})}\right)\bigg\}\bigg].
\end{align}
where to obtain the third inequality we note that $R<\frac{3d}{d-1}$.

Thus, using ($\ref{l62}$) in the last inequality above, we obtain
\begin{align}
\nonumber (II)&\lesssim \eta\sum_{N_1=\frac{2N}{8}}^{N_0} N_1^{\frac{2d}{R}-d+2}\lVert u_{N_1}\rVert_{L_t^\infty L_x^R}\\
&=\eta\sum_{N_1=\frac{2N}{8}}^{N_0} N_1^{\frac{d}{R}-d+3}\mathcal{S}(N_1).
\label{l70}
\end{align}

\underline{\it Term $(III)_{i,j}$}: 
Fix $i,j\in \{1,2,3\}$.  Using H\"older's inequality followed by the Bernstein and Sobolev inequalities, we get
\begin{align}
\nonumber \lVert u_{>N_0}u_iu_j\rVert_{L_t^\infty L_x^{R'}}&\leq \lVert u_{>N_0}\rVert_{L_t^\infty L_x^\frac{dR'}{d-2R'}}\lVert u_{i}\rVert_{L_t^\infty L_x^d}\lVert u_j\rVert_{L_t^\infty L_x^d}\\
\nonumber &\lesssim N_0^{3-\frac{d}{R'}}\lVert |\nabla|^{\frac{d}{R'}-3}u_{>N_0}\rVert_{L_t^\infty L_x^\frac{dR'}{d-2R'}}\lVert u\rVert_{L_t^\infty L_x^d}^{2}\\
\nonumber &\lesssim N_0^{3-\frac{d}{R'}}\lVert |\nabla|^{s_c}u_{>N_0}\rVert_{L_t^\infty L_x^2}\lVert |\nabla|^{s_c}u\rVert_{L_t^\infty L_x^2}^2\\
&\lesssim N_0^{3+\frac{d}{R}-d}\label{l71}
\end{align}
where in passing from the second line to the third line, we use $R<\frac{2d}{d-4}$, and in the last inequality we observed that 
\begin{align*}
\lVert |\nabla|^{s_c}u_{>N_0}\rVert_{L_t^\infty L_x^2}\leq \lVert (u,u_t)\rVert_{L_t^\infty(\mathbb{R};\dot{H}_x^{s_c}\times \dot{H}_x^{s_c-1})}.
\end{align*}

\underline{\it Term $(IV)_i$}:
Fix $i\in \{1,2\}$.  By H\"older's inequality, together with the Sobolev and Bernstein inequalities, we have
\begin{align}
\nonumber &\lVert u_{\frac{N}{8}<\cdot\leq N_0}u_{\leq \frac{N}{8}}u_i\rVert_{L_t^\infty L_x^{R'}}\\
\nonumber &\hspace{0.2in}\leq \lVert u_{\frac{N}{8}<\cdot \leq N_0}\rVert_{L_t^\infty L_x^2}\lVert u_{\leq \frac{N}{8}}\rVert_{L_t^\infty L_x^\frac{2Rd}{(d-2)R-2d}}\lVert u_{i}\rVert_{L_t^\infty L_x^d}\\
\nonumber &\hspace{0.2in}\lesssim \lVert P_{>\frac{N}{8}}P_{\leq N_0}u\rVert_{L_t^\infty L_x^2}\lVert u_{\leq \frac{N}{8}}\rVert_{L_t^\infty L_x^\frac{2Rd}{(d-2)R-2d}}\lVert |\nabla|^{s_c}u\rVert_{L_t^\infty L_x^2}\\
\nonumber &\hspace{0.2in}\lesssim \left(\tfrac{N}{8}\right)^{-s_c}\lVert |\nabla|^{s_c}u_{\frac{N}{8}<\cdot\leq N_0}\rVert_{L_t^\infty L_x^2}\sum_{N_1\leq \frac{N}{8}} \lVert u_{N_1}\rVert_{L_t^\infty L_x^\frac{2Rd}{(d-2)R-2d}}\\
\label{l72}&\hspace{0.2in}\lesssim N^{1-\frac{d}{2}}\eta\sum_{N_1\leq \frac{N}{8}}\lVert u_{N_1}\rVert_{L_t^\infty L_x^\frac{2Rd}{(d-2)R-2d}}\\
\label{l73} &\hspace{0.2in}\lesssim N^{1-\frac{d}{2}}\eta\sum_{N_1\leq \frac{N}{8}} N_1^{\frac{d}{R}-\frac{(d-2)R-2d}{2R}}\lVert u_{N_1}\rVert_{L_t^\infty L_x^R}\\
\nonumber &\hspace{0.2in}\lesssim N^{1-\frac{d}{2}}\eta \sum_{N_1\leq \frac{N}{8}} N_1^{\frac{2d}{R}-\frac{d}{2}+1}N_1^{1-\frac{d}{R}}\mathcal{S}(N_1)\\
\label{l74}&\hspace{0.2in}=N^{\frac{d}{R}-d+3}\eta\sum_{N_1\leq \frac{N}{8}} \left(\tfrac{N_1}{N}\right)^{\frac{d}{R}-\frac{d}{2}+2}.
\end{align}
where to obtain ($\ref{l72}$) we note that $N\leq 8N_0$ and to obtain ($\ref{l73}$) we used $R<\frac{3d}{d-1}$.

Collecting the estimates (\ref{l68}), (\ref{l69}), (\ref{l70}), (\ref{l71}) and (\ref{l74}), we obtain the desired inequality ($\ref{l64}$).

To obtain ($\ref{l65}$), we invoke Lemma $\ref{l107}$ in Appendix A.  This is a version of Gronwall's inequality which we recall from \cite{KillipVisanSupercriticalNLW3D}.  In particular, we define $x_k=\mathcal{S}(2^{-k}N_0)$, $k\in\mathbb{N}$ and note that ($\ref{l64}$) combined with Lemma \ref{l107} gives the bound
\begin{align}
x_k\lesssim 2^{-k\rho}\label{l75}
\end{align}
for each $\rho \in (0,d-\frac{d}{R}-3)$.  For the details in obtaining the bound $(\ref{l75})$ we refer the reader to Appendix A.  Thus, for each $N=2^{-k}N_0\leq 8N_0$ we obtain
\begin{align*}
\mathcal{S}(N)=\mathcal{S}(2^{-k}N_0)\lesssim (2^{-k})^\rho\sim N^\rho.
\end{align*}
Taking $\rho=\frac{d-4}{2}$ gives the desired bound ($\ref{l65}$).
\end{proof}

With this lemma in hand, we are now ready to prove Lemma $\ref{l58}$:
\begin{proof}[Proof of Lemma \ref{l58}.]
Recalling the definition of $\mathcal{S}(N)$, ($\ref{l65}$) shows that for all $N\leq 8N_0$,
\begin{align}
\lVert u_N\rVert_{L_t^\infty L_x^R}\lesssim N^{\frac{d}{2}-\frac{d}{R}-1}.\label{l76}
\end{align}

Then, using ($\ref{l76}$) along with the Bernstein inequalities, we obtain
\begin{align*}
\lVert u\rVert_{L_t^\infty L_x^R}&\leq \lVert u_{\leq N_0}\rVert_{L_t^\infty L_x^R}+\lVert u_{>N_0}\rVert_{L_t^\infty L_x^R}\\
&\lesssim \sum_{N\leq N_0} \lVert u_N\rVert_{L_t^\infty L_x^R}+\sum_{N>N_0}N^{\frac{d}{2}-\frac{d}{R}}\lVert u_{N}\rVert_{L_t^\infty L_x^2}\\
&\lesssim \sum_{N\leq N_0} N^{\frac{d}{2}-\frac{d}{R}-1}+\sum_{N>N_0}N^{1-\frac{d}{R}}\lVert |\nabla|^{s_c}u\rVert_{L_t^\infty L_x^2}\\
&\lesssim 1,
\end{align*}
where we note that our hypotheses on $d$ and $R$ ensure that $\frac{d}{2}-\frac{d}{R}-1>0$ and $1-\frac{d}{R}<0$.  Since $R$ is arbitrary, we obtain the lemma for every $q_0\in (\frac{2(d-1)}{d-3},\min\{\frac{2d}{d-4},\frac{3d}{d-1}\})$.  

We note that the lemma then follows for every $q_0\in (\frac{2(d-1)}{d-3},d]$ by using interpolation with the $L_t^\infty L_x^d$ bound which results from combining the a priori bound $u\in L_t^\infty\dot{H}_x^{s_c}$ with the Sobolev embedding.
\end{proof}

\subsection{Proof of Lemma \ref{l60}.}\ \\
\hspace*{0.1in} Let $u, q_1$ and $s$ be given as stated in the lemma and choose $s_0\in (0,\frac{2(d-q_1)}{q_1})$.  Applying the Bernstein inequalities, we argue as follows: 
\begin{align}
\nonumber &\lVert |\nabla|^{s-s_0} u\rVert_{L_t^\infty L_x^2}+\lVert |\nabla|^{s-1-s_0} u_t\rVert_{L_t^\infty L_x^2}\\
\nonumber &\hspace{0.2in}\leq \sum_{N\leq 1} \lVert |\nabla|^{s-s_0} u_N\rVert_{L_t^\infty L_x^2}+\lVert |\nabla|^{s-1-s_0}\partial_t u_N\rVert_{L_t^\infty L_x^2}\\
\nonumber &\hspace{0.4in}+\sum_{N>1} \lVert |\nabla|^{s-s_0} u_N\rVert_{L_t^\infty L_x^2}+\lVert |\nabla|^{s-1-s_0} \partial_t u_N\rVert_{L_t^\infty L_x^2}\\
\nonumber &\hspace{0.2in}\lesssim \sum_{N\leq 1} N^{-s_0}\bigg[\lVert |\nabla|^s u_N\rVert_{L_t^\infty L_x^2}+\lVert |\nabla|^{s-1}\partial_t u_N\rVert_{L_t^\infty L_x^2}\bigg]\\
\nonumber &\hspace{0.4in}+\sum_{N>1} N^{s-s_0-s_c}\bigg[\lVert |\nabla|^{s_c} u_N\rVert_{L_t^\infty L_x^2}+\lVert |\nabla|^{s_c-1}\partial_tu_N\rVert_{L_t^\infty L_x^2}\bigg]\\
\nonumber &\hspace{0.2in}\lesssim \sum_{N\leq 1} N^{-s_0} \bigg[\lVert |\nabla|^s u_N\rVert_{L_t^\infty L_x^2}+\lVert |\nabla|^{s-1}\partial_tu_N\rVert_{L_t^\infty L_x^2}\bigg]+\sum_{N>1} N^{s-s_0-s_c}\\
\label{l77}&\hspace{0.2in}\lesssim \sum_{N\leq 1} N^{-s_0} \bigg[\lVert |\nabla|^s u_N\rVert_{L_t^\infty L_x^2}+\lVert |\nabla|^{s-1}\partial_t u_N\rVert_{L_t^\infty L_x^2}\bigg]+1
\end{align}
where we note $\lVert (u,u_t)\rVert_{L_t^\infty(\dot{H}_x^{s_c}\times \dot{H}_x^{s_c-1})}\leq C$ to obtain the third inequality followed by $\displaystyle \sum_{N>1} N^{s-s_0-s_c}<\infty$ for $s-s_0-s_c<0$ to obtain the fourth inequality.

To obtain ($\ref{l61}$), it thus remains to estimate the term $\lVert |\nabla|^s u_N\rVert_{L_t^\infty L_x^2}+\lVert |\nabla|^{s-1} \partial_t u_N\rVert_{L_t^\infty L_x^2}$ in ($\ref{l77}$).  
We begin by noting that the unitary property of the linear propagator $\mathcal{W}(\cdot)$ implies that for every $t_1,t_2\in\mathbb{R}$ and $g,h\in L^2$,
\begin{align*}
&\langle |\nabla|\tfrac{\sin(t_1|\nabla|)}{|\nabla|}g,-|\nabla|\tfrac{\sin(t_2|\nabla|)}{|\nabla|}h\rangle+\langle \cos(t_1|\nabla|)g,-\cos(t_2|\nabla|)h\rangle\\
&\hspace{0.9in}=\langle g,-\cos((t_1-t_2)|\nabla|)h\rangle,
\end{align*}
Next, without loss of generality we take $t=0$, and note that by using the above observation and Lemma $\ref{l17}$ we write
\begin{align}
\nonumber &\lVert |\nabla|^s u_N(0)\rVert_{L_x^2}^2+\lVert |\nabla|^{s-1} \partial_t u_N(0)\rVert_{L_x^2}^2\\
\nonumber &\hspace{0.1in}=\lim_{T\rightarrow\infty}\lim_{T'\rightarrow -\infty} \langle |\nabla|\int_0^T \tfrac{\sin(-t'|\nabla|)}{|\nabla|}P_N|\nabla|^{s-1}F(u(t'))dt',\\
\nonumber &\hspace{1.4in}-|\nabla| \int_{T'}^0 \tfrac{\sin(-\tau'|\nabla|)}{|\nabla|}P_N|\nabla|^{s-1}F(u(\tau'))d\tau'\rangle\\
\nonumber &\hspace{0.5in}+\langle \int_0^T \cos(-t'|\nabla|)P_N|\nabla|^{s-1}F(u(t'))dt',\\
\nonumber &\hspace{1.4in}-\int_{T'}^0 \cos(-\tau'|\nabla|)P_N|\nabla|^{s-1}F(u(\tau'))d\tau'\rangle\\
\label{l78}&\hspace{0.1in}\leq \int_0^\infty \int_{-\infty}^0 \bigg|\langle P_N|\nabla|^{s-1}F(u(t')),-\cos((t'-\tau')|\nabla|)P_N|\nabla|^{s-1}F(u(\tau'))\rangle\bigg|d\tau'dt'
\end{align}

Setting $r=\frac{2q_1}{q_1+4}$ and using H\"older's inequality followed by Proposition $\ref{l4}$ and Bernstein's inequalities, we obtain
\begin{align}
\nonumber &\big|\langle P_N|\nabla|^sF(u(t')),\tfrac{\cos((t'-\tau')|\nabla|)}{|\nabla|^2}P_N|\nabla|^sF(u(\tau'))\rangle\big|\\
\nonumber &\hspace{0.4in}\lesssim \lVert P_N|\nabla|^sF(u(t'))\rVert_{L_x^{r}}\lVert \tfrac{\cos((t'-\tau')|\nabla|)}{|\nabla|^2}P_N|\nabla|^sF(u(\tau'))\rVert_{L_x^{r'}}\\
\nonumber &\hspace{0.4in}\lesssim \tfrac{1}{|t'-\tau'|^{(d-1)\left(\frac{1}{2}-\frac{1}{r'}\right)}}\lVert P_N|\nabla|^sF(u(t'))\rVert_{L_x^{r}}\lVert |\nabla|^{\frac{d-3}{2}-\frac{d+1}{r'}}P_N|\nabla|^sF(u(\tau'))\rVert_{L_x^{r}}\\
&\hspace{0.4in}\lesssim \tfrac{N^{\frac{d-3}{2}-\frac{d+1}{r'}}}{|t'-\tau'|^{(d-1)\left(\frac{1}{2}-\frac{1}{r'}\right)}}\lVert P_N|\nabla|^s F(u(t'))\rVert_{L_t^\infty L_x^{r}}^2\label{l79}
\intertext{On the other hand, using the Cauchy-Schwarz inequality followed by Proposition $\ref{l4}$ (with $p=2$) and Bernstein's inequality, we obtain}
\nonumber &\big|\langle P_N|\nabla|^sF(u(t')),\tfrac{\cos((t'-\tau')|\nabla|)}{|\nabla|^2}P_N|\nabla|^sF(u(\tau'))\rangle\big|\\
\nonumber &\hspace{0.4in}\lesssim \lVert P_N|\nabla|^sF(u(t'))\rVert_{L_x^2}\lVert \tfrac{\cos((t'-\tau')|\nabla|)}{|\nabla|^2}P_N|\nabla|^sF(u(\tau'))\rVert_{L_x^2}\\
\nonumber &\hspace{0.4in}\lesssim \lVert P_N|\nabla|^sF(u)\rVert_{L_x^2}\lVert |\nabla|^{-2}P_N|\nabla|^sF(u)\rVert_{L_x^2}\\
\nonumber &\hspace{0.4in}\lesssim N^{-2}\lVert P_N|\nabla|^sF(u)\rVert_{L_x^2}^2\\
\label{l80}&\hspace{0.4in}\lesssim N^{-2+\frac{2d}{r}-d}\lVert |\nabla|^s F(u)\rVert_{L_t^\infty L_x^{r}}^2,
\end{align}
where we recall that $r<\frac{2d}{d+4}<2$.

Invoking the bounds $(\ref{l79})$ and $(\ref{l80})$ in $(\ref{l78})$ and using Lemma $\ref{l19}$, we obtain
\begin{align}
\nonumber &\lVert |\nabla|^s u_N(0)\rVert_{L_x^2}^2+\lVert |\nabla|^{s-1}\partial_t u_N(0)\rVert_{L_x^2}^2\\
\nonumber &\hspace{0.2in}\leq \lVert |\nabla|^sF(u)\rVert_{L_t^\infty L_x^{r}}^2\int_0^\infty \int_{-\infty}^0 \min \{\tfrac{N^{\frac{d-3}{2}-\frac{d+1}{r'}}}{|t'-\tau'|^{(d-1)(\frac{1}{2}-\frac{1}{r'})}},N^{-2+d-\frac{2d}{r'}}\}  dt'd\tau'\\
\nonumber &\hspace{0.2in}\leq \lVert |\nabla|^su\rVert^2_{L_t^\infty L_x^2}\lVert u\rVert_{L_t^\infty L_x^{q_1}}^4\int_0^\infty\int_{-\infty}^0 \min\{\tfrac{N^{\frac{d-3}{2}-\frac{d+1}{r'}}}{|t'-\tau'|^{(d-1)(\frac{1}{2}-\frac{1}{r'})}},N^{-2+d-\frac{2d}{r'}}\}dt'd\tau'\\
\nonumber &\hspace{0.2in}=N^{-2+d-\frac{2d}{r'}}\lVert |\nabla|^su\rVert^2_{L_t^\infty L_x^2}\lVert u\rVert_{L_t^\infty L_x^{q_1}}^4\int_0^\infty\int_{-\infty}^0 \min\{\tfrac{N^{-(d-1)}}{|t'-\tau'|^{d-1}},1 \}^{\frac{1}{2}-\frac{1}{r'}}dt'd\tau'
\end{align}
We conclude the proof by estimating the above integral.  To this end, we use the bound
\begin{align}
\int_0^\infty \int_{-\infty}^0 \min\{\tfrac{N^{-(d-1)}}{|t'-\tau'|^{d-1}},1\}^{\frac{1}{2}-\frac{1}{r'}}dt'd\tau'&\lesssim N^{-2},
\end{align}
which follows from the assumption $q_1<d-1$ and a straightforward computation.

Invoking this bound in ($\ref{l77}$) and using the hypotheses $u\in L_t^\infty L_x^{q_1}$ and $|\nabla|^su\in L_t^\infty L_x^2$, we get
\begin{align*}
&\lVert |\nabla|^{s-s_0}u\rVert_{L_t^\infty L_x^2}+\lVert |\nabla|^{s-s_0-1}u_t\rVert_{L_t^\infty L_x^2}\\
&\hspace{0.8in}\lesssim \sum_{N\leq 1} N^{-s_0}N^{-2+\frac{d}{2}-\frac{d}{r'}}+1\\
&\hspace{0.8in}=\sum_{N\leq 1} N^{\frac{2d}{q}-2-s_0}+1
\end{align*}

Note that by our choice of $s_0$, we have $\frac{2d}{q_1}-2-s_0>0$, so that the desired bound ($\ref{l61}$) holds.
\qed
\section{Soliton-like solution}
In this section, we rule out the second blow-up scenario identified in Theorem $\ref{l15}$, the soliton-like solution.  

As in \cite{KillipVisanSupercriticalNLS,KillipVisanSupercriticalNLW3D}, our approach to obtain the desired contradiction is to get an upper and lower bound on the quantity 
\begin{align}
\int_{I}\int_{\mathbb{R}^d} \frac{|u(t,x)|^4}{|x|}dxdt,
\label{l81}
\end{align}
with a time interval $I\subset\mathbb{R}$.  Indeed, the Morawetz estimate (Theorem $\ref{l7}$) and the additional decay property given in Theorem $\ref{l56}$ immediately imply that ($\ref{l81}$) is bounded from above independent of $I$.  The contradiction will then follow once we obtain a lower bound on ($\ref{l81}$) which grows to infinity as $|I|\rightarrow\infty$.  

We obtain the lower bound in two steps: the first step is to get an estimate on the growth of $x(t)$ via the finite speed of propagation in the form of Lemma $\ref{l37}$.  The second step is then to show that the $L_{t,x}^4$ norm of $u$ over unit time intervals and localized in space near $x(t)$ is bounded away from zero. 

The key ingredient used to control $x(t)$ in Step $1$ is to obtain a bound from below in a suitable space for all times.  This requires the additional decay result, Theorem $\ref{l56}$.
\begin{lemma}
\label{l82}
Suppose that $u:\mathbb{R}\times \mathbb{R}^d\rightarrow\mathbb{R}$ is a solution to (NLW) which satisfies the properties of a soliton-like solution stated in Theorem $\ref{l15}$.  Then there exists $\eta>0$ such that for all $t\in\mathbb{R}$, 
\begin{align*}
\int_{\mathbb{R}^d} |u(t,x)|^d+|u_t(t,x)|^\frac{d}{2}dx\geq \eta.
\end{align*}
\end{lemma}

\begin{proof}
Suppose to the contrary that the claim failed.  Then there exists a sequence $\{t_n\}\subset\mathbb{R}$ such that
\begin{align}
(u(t_n),u_t(t_n))\rightarrow (0,0)\quad\textrm{in}\quad L^d_x\times L_x^{\frac{d}{2}}
\label{l83}
\end{align}
as $n\rightarrow\infty$.
Since $u$ is a soliton-like solution, $\{(u(t_n,x(t_n)+\cdot),u_t(t_n,x(t_n)+\cdot)):n\in\mathbb{N}\}$ has compact closure in $\dot{H}_x^{s_c}\times \dot{H}_x^{s_c-1}$.

Note that by the precompactness of $\{u(t_n,x(t_n)+\cdot),u_t(t_n,x(t_n)+\cdot):n\in\mathbb{N}\}$ there exists a subsequence (still indexed by $n$) such that $(u(t_n,x(t_n)+\cdot),u_t(t_n,x(t_n)+\cdot))\rightarrow (u_0^*,u_1^*)$ in $\dot{H}_x^{s_c}\times \dot{H}_x^{s_c-1}$.  However, $(\ref{l83})$ and the change of variable $x\mapsto x(t_n)+x$ imply $(u(t_n,x(t_n)+\cdot),u_t(t_n,x(t_n)+\cdot))\rightarrow (0,0)$ in $L_x^d\times L_x^\frac{d}{2}$, so that the continuous embedding $\dot{H}_x^{s_c}\times\dot{H}_x^{s_c-1}\hookrightarrow L_x^d\times L_x^\frac{d}{2}$ and the uniqueness of limits give $(u_0^*,u_1^*)=(0,0)$.  Thus by the change of variable $x\mapsto -x(t_n)+x$, we have
\begin{align}
\label{l84}(u(t_n),u_t(t_n))\rightarrow (0,0) \quad\textrm{ in}\quad \dot{H}_x^{s_c}\times \dot{H}_x^{s_c-1}.
\end{align}

We now note that for all $n\in\mathbb{N}$, if $\epsilon$ is as in Theorem $\ref{l56}$, then there exist $\theta_1,\theta_2,\theta_3\in (0,1)$ such that
\begin{align}
\nonumber E(u_0,u_1)&=E(u(t_n),u_t(t_n))\\
\nonumber &=\frac{1}{2}\int_{\mathbb{R}^d} |\nabla u(t_n)|^2dx+\frac{1}{2}\int_{\mathbb{R}^d} |u_t(t_n)|^2dx+\frac{1}{4}\int_{\mathbb{R}^d} |u(t_n)|^4dx\\
\nonumber  &\lesssim \lVert u(t_n)\rVert_{\dot{H}_x^{s_c}}^{\theta_1}\lVert (u,u_t)\rVert_{L_t^\infty(\mathbb{R};\dot{H}_x^{1-\epsilon}\times \dot{H}_x^{-\epsilon})}^{1-\theta_1}\\
\nonumber &\hspace{0.2in}+\lVert u_t(t_n)\rVert_{\dot{H}_x^{s_c-1}}^{\theta_2}\lVert (u,u_t)\rVert^{1-\theta_2}_{L_t^\infty(\mathbb{R};\dot{H}_x^{1-\epsilon}\times \dot{H}_x^{-\epsilon})}\\
\nonumber &\hspace{0.2in}+\lVert u(t_n)\rVert_{\dot{H}_x^{s_c}}^{4\theta_3}\lVert (u,u_t)\rVert_{L_t^\infty(\mathbb{R};\dot{H}_x^{1-\epsilon}\times\dot{H}_x^{-\epsilon})}^{4(1-\theta_3)}.
\end{align}
where in obtaining the inequality we used $\lVert u(t_n)\rVert_{L_x^4}\lesssim \lVert u(t_n)\rVert_{\dot{H}_x^\frac{d}{4}}$ and interpolation.

Letting $n\rightarrow\infty$ and applying ($\ref{l84}$) followed by the conservation of energy, we obtain
\begin{align*}
E(u_0,u_1)=0.
\end{align*}
Thus $u\equiv 0$, contradicting our assumption that $\lVert u\rVert_{L_{t,x}^{d+1}}=\infty$.
\end{proof}

Based on the previous lemma and the finite speed of propagation in the sense of Lemma $\ref{l37}$, we now prove the following estimate for $x(t)$:
\begin{lemma}
\label{l85}
Suppose that $u:\mathbb{R}\times \mathbb{R}^d\rightarrow\mathbb{R}$ is a solution to (NLW) which satisfies the properties of a soliton-like solution stated in Theorem $\ref{l15}$.  Then there exists $C>0$ such that for every $t\geq 0$ we have,
\begin{align*}
|x(t)-x(0)|\leq C+t.
\end{align*}
\end{lemma}

\begin{proof}
We argue in a similar spirit to \cite{KillipVisanSupercriticalNLW3D}.  Fix $\eta>0$ to be determined later in the argument.  Let us first note that by Remark $\ref{l13}$ there exists $c(\eta)>0$ such that
\begin{align}
\int_{|x-x(t)|>c(\eta)} |u(t,x)|^ddx+\int_{|x-x(t)|>c(\eta)} |u_t(t,x)|^\frac{d}{2}dx\leq \eta\label{l86}
\end{align}
for all $t\in\mathbb{R}$.

Next, applying Lemma $\ref{l37}$ with $\epsilon=\eta$ and $t=0$, we choose $R>0$ such that for all $r\in\mathbb{R}$ and $x\in \{x\in\mathbb{R}^d:|x|\geq 2R+r\}$ we have 
\begin{align*}
v^{(0)}(r,x)=v_{R}^{(0)}(r,x)
\end{align*}
where $v^{(t)}$ and $v_{R}^{(t)}$ are defined as in Lemma $\ref{l37}$.

\noindent Then, for all $t\in\mathbb{R}$, we obtain
\begin{align}
\nonumber &\int_{|x-x(0)|>2R+t} |u(t,x)|^ddx+\int_{|x-x(0)|>2R+t} |u_t(t,x)|^\frac{d}{2}dx\\
\nonumber &\hspace{0.2in}=\int_{|x|>2R+t} |u(t,x+x(0))|^ddx+\int_{|x|>2R+t} |u_t(t,x+x(0))|^\frac{d}{2}dx\\
\nonumber &\hspace{0.2in}=\int_{|x|>2R+t} |v^{(0)}(t,x)|^ddx+\int_{|x|>2R+t} |\partial_t v^{(0)}(t,x)|^\frac{d}{2}dx\\
\nonumber &\hspace{0.2in}=\int_{|x|>2R+t} |v_{R}^{(0)}(t,x)|^ddx+\int_{|x|>2R+t} |\partial_t v_{R}^{(0)}(t,x)|^\frac{d}{2}dx\\
\nonumber &\hspace{0.2in}\leq \int_{\mathbb{R}^d} |v_{R}^{(0)}(t,x)|^ddx+\int_{\mathbb{R}^d} |\partial_tv_{R}^{(0)}(t,x)|^\frac{d}{2}dx\\
\nonumber &\hspace{0.2in}\leq \left(\int_{\mathbb{R}^d} ||\nabla|^{s_c}v_{R}^{(0)}(t,x)|^2dx\right)^{d/2}+\left(\int_{\mathbb{R}^d} ||\nabla|^{s_c-1}\partial_t v_{R}^{(0)}(t,x)|^2dx\right)^{d/4}\\
\nonumber &\hspace{0.2in}\leq (C\eta)^d+(C\eta)^{d/2}\\
&\hspace{0.2in}\leq C,\eta\label{l87}
\end{align}
where in the second to last inequality we used the smallness given by ($\ref{l38}$) in Lemma $\ref{l37}$.

Combining the bounds $(\ref{l86})$ and $(\ref{l87})$, we obtain 
\begin{align}
\nonumber &\int_{\{x:|x-x(t)|\geq c(\eta)\}\cup \{x:|x-x(0)|\geq 2R+t\}} |u(t,x)|^d+|u_t(t,x)|^\frac{d}{2}dx\\
\nonumber &\hspace{0.2in}\leq \int_{|x-x(t)|\geq c(\eta)} |u(t,x)|^d+|u_t(t,x)|^\frac{d}{2}dx+\int_{|x-x(0)|\geq 2R+t} |u(t,x)|^d+|u_t(t,x)|^\frac{d}{2}dx\\
\label{l88}&\hspace{0.2in}\leq (1+C)\eta.\\
\intertext{
for all $t\geq0$.  We now determine $\eta$.  Note that by Lemma $\ref{l82}$ together with the assumption $(u,u_t)\in L_t^\infty(\mathbb{R};\dot{H}_x^{s_c}\times\dot{H}_x^{s_c-1})$, we have
\begin{align*}
0<\inf_{t\in\mathbb{R}} \bigg(\lVert u(t)\rVert_{L_x^d}^d+\lVert u_t(t)\rVert_{L_x^\frac{d}{2}}^\frac{d}{2}\bigg)<\infty,
\end{align*}
so that we may choose $\eta>0$ such that
\begin{align*}
\eta<\frac{1}{4(1+C)}\inf_{t\in\mathbb{R}}\bigg(\lVert u(t)\rVert_{L_x^d}^d+\lVert u_t(t)\rVert_{L_x^\frac{d}{2}}^\frac{d}{2}\bigg).
\end{align*}
Thus invoking this choice of $\eta$ in ($\ref{l88}$), we have for all $t\geq 0$,
}
\nonumber &\int_{\{x:|x-x(t)|<c(\eta)\}\cap \{x:|x-x(0)|<2R+t\}} |u(t,x)|^d+|u_t(t,x)|^\frac{d}{2}dx\\
\nonumber &\hspace{0.2in}=\int_{\mathbb{R}^d} |u(t,x)|^d+|u_t(t,x)|^\frac{d}{2}dx\\
\nonumber &\hspace{0.3in}-\int_{\{x:|x-x(t)|\geq c(\eta)\}\cup \{x:|x-x(0)|\geq 2R+t\}} |u(t,x)|^d+|u_t(t,x)|^\frac{d}{2}dx\\
\nonumber &\hspace{0.2in}\geq \inf_{t\in\mathbb{R}}\left(\lVert u(t)\rVert_{L_x^d}^d+\lVert u_t(t)\rVert_{L_x^\frac{d}{2}}^\frac{d}{2}\right)-(1+C)\eta\\
\nonumber &\hspace{0.2in}\geq \left(1-\frac{1}{4}\right)\inf_{t\in\mathbb{R}}\left(\lVert u(t)\rVert_{L_x^d}^d+\lVert u_t(t)\rVert_{L_x^\frac{d}{2}}^\frac{d}{2}\right)\\
\nonumber &\hspace{0.2in}>0.
\end{align}

Thus, we conclude that for all $t\geq 0$, the set
\begin{align*}
X(t)=\{x:|x-x(t)|<c(\eta)\}\cap \{x:|x-x(0)|<2R+t\}\neq \emptyset.
\end{align*}
We may then choose $x\in X(t)$, $t\geq 0$, so that
\begin{align*}
|x(t)-x(0)|&\leq |x(t)-x|+|x-x(0)|\leq c(\eta)+2R+t.
\end{align*}
Noting that $\eta$ and $R$ are independent of $t$, we conclude that there exists $C>0$ such that for all $t\geq 0$ we have
\begin{align*}
|x(t)-x(0)|\leq C+t
\end{align*}
as desired.
\end{proof}

The second step in obtaining the lower bound on $(\ref{l81})$ is the following lemma which employs the almost periodicity as well as the dispersive estimate.
\begin{lemma}
\label{l89}
Suppose that $u:\mathbb{R}\times\mathbb{R}^d\rightarrow\mathbb{R}$ is a solution to (NLW) which satisfies the properties of a soliton-like solution stated in Theorem $\ref{l15}$.  Then there exists $R>0$ and $c>0$ such that for every $s\in \mathbb{R}$,
\begin{align}
\int_{s}^{s+1}\int_{|x-x(t)|\leq R} |u(t,x)|^4dxdt\geq c.
\label{l90}
\end{align}
\end{lemma}

\begin{proof}
We argue in a similar manner as in \cite{KillipVisanSupercriticalNLW3D}.  As a first step, we claim that there exists $C_1>0$ such that for every $s\in\mathbb{R}$,
\begin{align}
\label{l91}
\bigg|\bigg\{t\in [s,s+1]:\int_{\mathbb{R}^d} |u(t)|^\frac{2d}{d-2}dx\geq C_1\bigg\}\bigg|\geq C_1.
\end{align}

To this end, suppose to the contrary that the claim failed.  Then there exists a sequence of times $\{s_n\}\subset\mathbb{R}$ such that for every $n\in\mathbb{N}$,
\begin{align*}
\bigg|\bigg\{\tau\in [0,1]:\int_{\mathbb{R}^d} |u(s_{n}+\tau)|^\frac{2d}{d-2}dx\geq \frac{1}{n}\bigg\}\bigg|<\frac{1}{n}.
\end{align*}
This in turn implies that the sequence $g_n:[0,1]\rightarrow\mathbb{R}$ defined by 
\begin{align*}
g_n(\tau)=\int_{\mathbb{R}^d} |u(s_n+\tau)|^\frac{2d}{d-2}dx 
\end{align*}
converges to zero in measure as $n\rightarrow\infty$.  We next extract a subsequence (still labeled $s_n$) such 
that 
\begin{align}
\label{l92}
\int_{\mathbb{R}^d} |u(s_n+\tau)|^\frac{2d}{d-2}dx\rightarrow 0\quad\textrm{ for a.e. }\quad \tau\in [0,1]\quad\textrm{ as }\quad n\rightarrow\infty.
\end{align}

To continue, using the hypothesis that $u$ is a soliton-like solution together with the almost periodicity of $u$, we choose a further subsequence (still labeled $s_n$) and a pair $(f,g)\in \dot{H}_x^{s_c}\times\dot{H}_x^{s_c-1}$ such that
\begin{align}
\label{l93}
(u(s_n,x(s_n)+\cdot),u_t(s_n,x(s_n)+\cdot)\rightarrow (f,g)\quad \textrm{in}\quad \dot{H}_x^{s_c}\times\dot{H}_x^{s_c-1}.
\end{align}

\noindent Moreover, using the additional decay property (Theorem $\ref{l56}$) we observe that the sequence $\{(u(s_n,x(s_n)+\cdot),u_t(s_n,x(s_n)+\cdot))\}$ is bounded in $\dot{H}_x^1\times L_x^2$, and we therefore pass to another subsequence to find $(f',g')\in \dot{H}_x^1\times L_x^2$ such that 
\begin{align}
(u(s_n,x(s_n)+\cdot),u_t(s_n,x(s_n)+\cdot))\rightharpoonup (f',g')\quad\textrm{ weakly in }\quad\dot{H}_x^{1}\times L_x^{2}.
\label{l94}
\end{align}

Next, we show that we have $(f'(x),g'(x))=(0,0)$ for a.e. $x\in\mathbb{R}^d$.  To prove this, we begin by noting that it suffices to show 
\begin{align}
\label{l95}
\mathcal{W}(\tau)(f',g')(x)=0,\quad\textrm{ for a.e. }\quad \tau\in [0,1]\quad\textrm{ and a.e. }\quad x\in\mathbb{R}^d.
\end{align}
Indeed, if we assume ($\ref{l95}$), then in view of $\mathcal{W}(\tau)(f',g')\in C_\tau^0(\dot{H}_x^1)\cap C_\tau^1(L_x^2)$, we obtain
\begin{align*}
\lVert f'\rVert_{L_x^{\frac{2d}{d-2}}}&\lesssim \lVert f'\rVert_{\dot{H}_x^1}=\lim_{\tau\rightarrow 0} \lVert \mathcal{W}(\tau)(f',g')\rVert_{\dot{H}_x^1}=0,
\end{align*}
as well as
\begin{align*}
\lVert g'\rVert_{L_x^2}&=\lim_{\tau\rightarrow 0} \lVert \partial_\tau \mathcal{W}(\tau)(f',g')\rVert_{L_x^2}=\lim_{\tau\rightarrow 0}\lim_{h\rightarrow 0} \lVert \frac{1}{h}[\mathcal{W}(\tau+h)(f',g')-\mathcal{W}(\tau)(f',g')]\rVert_{L_x^2}=0.
\end{align*}

We now turn to verifying the assertion $(\ref{l95})$.  We first note that $(\ref{l94})$ yields $\mathcal{W}(\tau)(u(s_n),x(s_n)+\cdot),u_t(s_n,x(s_n)+\cdot))\rightharpoonup \mathcal{W}(\tau)(f',g')$ weakly in $L_x^\frac{2d}{d-2}$ for every $\tau\in \mathbb{R}$ (for a justification of this claim, we refer to Proposition \ref{l110} in Appendix A).  The weak lower semicontinuity of the norm then yields
\begin{align}
\lVert \mathcal{W}(\tau)(f',g')\rVert_{L_x^\frac{2d}{d-2}}&\leq \lim_{n\rightarrow\infty} \lVert \mathcal{W}(\tau)(u(s_n),u_t(s_n))\rVert_{L_x^\frac{2d}{d-2}} 
\label{l96}
\end{align}
for every $\tau\in \mathbb{R}$.

\noindent Fix $\tau\in [0,1]$.  Using the Duhamel formula, the dispersive estimate followed by Lemma $\ref{l19}$ twice, and the Sobolev embedding, we obtain for all $n\in\mathbb{N}$,
\begin{align}
\nonumber &\lVert \mathcal{W}(\tau)(u(s_n),u_t(s_n))\rVert_{L_x^\frac{2d}{d-2}}\\
\nonumber &\hspace{0.2in} \leq \lVert u(s_n+\tau)\rVert_{L_x^\frac{2d}{d-2}}+\int_{s_n}^{s_n+\tau} \lVert \frac{\sin((s_n+\tau-\tau')|\nabla|)}{|\nabla|}[u(\tau')]^3\rVert_{L_x^\frac{2d}{d-2}}d\tau'\\
\nonumber &\hspace{0.2in}\lesssim \lVert u(s_n+\tau)\rVert_{L_x^\frac{2d}{d-2}}+\int_{s_n}^{s_n+\tau} |s_n+\tau-\tau'|^{-\frac{d-1}{d}}\lVert |\nabla|^{\frac{1}{d}} [u(\tau')^3]\rVert_{L_x^\frac{2d}{d+2}}d\tau'\\
\nonumber &\hspace{0.2in}\lesssim \lVert u(s_n+\tau)\rVert_{L_x^\frac{2d}{d-2}}+\int_{s_n}^{s_n+\tau} |s_n+\tau-\tau'|^{-\frac{d-1}{d}}\lVert u(\tau')^2\rVert_{L_x^\frac{2d^2}{d^2-2}}\lVert |\nabla|^\frac{1}{d}u(\tau')\rVert_{L_x^\frac{d^2}{d+1}}d\tau'\\
\nonumber &\hspace{0.2in}\lesssim \lVert u(s_n+\tau)\rVert_{L_x^\frac{2d}{d-2}}+\int_{s_n}^{s_n+\tau} |s_n+\tau-\tau'|^{-\frac{d-1}{d}}\lVert u(\tau')\rVert^2_{L_x^\frac{4d^2}{d^2-2}}\lVert u(\tau')\rVert_{\dot{H}_x^{s_c}}d\tau'\\
\label{l97} &\hspace{0.2in}\lesssim \lVert u(s_n+\tau)\rVert_{L_x^\frac{2d}{d-2}}+\int_{s_n}^{s_n+\tau} |s_n+\tau-\tau'|^{-\frac{d-1}{d}}\lVert u(\tau')\rVert^2_{L_x^\frac{4d^2}{d^2-2}}d\tau'.
\end{align}

\noindent We estimate the above integral as follows:  Using interpolation, we deduce 
\begin{align}
\nonumber &\int_{s_n}^{s_n+\tau} |s_n+\tau-\tau'|^{-\frac{d-1}{d}}\lVert u(\tau')\rVert^2_{L_x^\frac{4d^2}{d^2-2}}d\tau'\\
\nonumber &\hspace{0.2in}=\int_0^\tau |\tau-\tau'|^{-\frac{d-1}{d}}\lVert u(s_n+\tau')\rVert^2_{L_x^\frac{4d^2}{d^2-2}}d\tau'\\
\nonumber &\hspace{0.2in}\lesssim \int_0^\tau |\tau-\tau'|^{-\frac{d-1}{d}}\lVert u(s_n+\tau')\rVert_{L_x^\frac{2d}{d-2}}^{2\theta}\lVert u(s_n+\tau')\rVert_{L_x^d}^{2(1-\theta)}d\tau'\\
\nonumber &\hspace{0.2in}\lesssim \int_0^\tau |\tau-\tau'|^{-\frac{d-1}{d}}\lVert u(s_n+\tau')\rVert_{L_x^\frac{2d}{d-2}}^{2\theta}\lVert u\rVert_{L^\infty_t\dot{H}_x^{s_c}}^{2(1-\theta)}d\tau'\\
&\hspace{0.2in}\lesssim \int_0^\tau |\tau-\tau'|^{-\frac{d-1}{d}}\lVert u(s_n+\tau')\rVert_{L_x^\frac{2d}{d-2}}^{2\theta}d\tau'
\end{align}
for some $\theta\in (0,1)$.  Then, by virtue of Theorem $\ref{l56}$ and ($\ref{l92}$), the dominated convergence theorem yields
\begin{align}
\int_0^\tau |\tau-\tau'|^{-\frac{d-1}{d}}\lVert u(s_n+\tau')\rVert_{L_x^\frac{2d}{d-2}}^{2\theta}d\tau'\rightarrow 0.
\label{l98}
\end{align}

\noindent Thus appealing to ($\ref{l92}$) once again, together with ($\ref{l98}$), we use $(\ref{l97})$ to obtain
\begin{align*}
\lVert \mathcal{W}(s)(u(s_n),u_t(s_n))\rVert_{L_x^\frac{2d}{d-2}}\rightarrow 0
\end{align*}
which in turn gives the claim ($\ref{l95}$) so that $f'(x)=g'(x)=0$ a.e. as claimed.

Now, note that by combining ($\ref{l93}$) and ($\ref{l94}$) with the Sobolev embedding and uniqueness of weak limits in $L_x^p$ spaces, we obtain $(f(x),g(x))=(f'(x),g'(x))$ for a.e. $x\in\mathbb{R}^d$.  Thus, using $(\ref{l93})$ with $f(x)=g(x)=0$ for a.e. $x\in\mathbb{R}^d$, we may choose $n$ so that $\lVert (u(s_n,x(s_n)+\cdot),u_t(s_n,x(s_n)+\cdot))\rVert_{\dot{H}_x^{s_c}\times \dot{H}_x^{s_c-1}}$ is arbitrarily small.  The local theory then gives $\lVert u\rVert_{L_{t,x}^{d+1}}<\infty$, contradicting our hypothesis that $u$ is a blow-up solution.  Thus ($\ref{l91}$) holds as desired.

Our second step is to adjust the domain of integration in ($\ref{l91}$).  To this end, let $C_1$ be as in ($\ref{l91}$).  Fix $\eta>0$ to be determined later in the argument and let $s\in\mathbb{R}$ be given.  Then, by the almost periodicity of $u$, we may choose $C_2(\eta)>0$ such that
\begin{align*}
\lVert u(t)\rVert_{L_x^d(|x-x(t)|\geq C_2(\eta))}\leq \eta^\frac{1}{d}.
\end{align*}
Let $\epsilon>0$ be as in Theorem $\ref{l56}$.  Using interpolation followed by the Sobolev embedding, we have
\begin{align}
\nonumber \lVert u(t)\rVert_{L_x^\frac{2d}{d-2}(|x-x(t)|\geq C_2(\eta))}&\leq \lVert u(t)\rVert_{L_x^d(|x-x(t)|\geq C_2(\eta))}^{\gamma}\lVert u(t)\rVert_{L_x^\frac{2d}{d-2(1-\epsilon)}(\mathbb{R}^d)}^{1-\gamma}\\
\nonumber &\leq C\lVert u(t)\rVert_{L_x^d(|x-x(t)|\geq C_2(\eta))}^{\gamma}\lVert u(t)\rVert_{\dot{H}_x^{1-\epsilon}(\mathbb{R}^d)}^{1-\gamma}\\
&\leq C\eta^\frac{\gamma}{d}\label{l99}
\end{align}
for some $\gamma\in (0,1)$, where we note that $d\geq 6$ yields $\frac{2d}{d-2(1-\epsilon)}<\frac{2d}{d-2}<d$.

Choose $\eta$ small enough so that $(C\eta^\frac{\gamma}{d})^\frac{2d}{d-2}<\frac{C_1}{2}$.  Then for all $t\in [s,s+1]$, $\int_{\mathbb{R}^d} |u(t,x)|^\frac{2d}{d-2}\geq C_1$ implies
\begin{align*}
\int_{|x-x(t)|\leq C_2(\eta)} |u(t,x)|^\frac{2d}{d-2}dx&=\int_{\mathbb{R}^d} |u(t,x)|^\frac{2d}{d-2}dx-\int_{|x-x(t)|\geq C_2(\eta)} |u(t,x)|^\frac{2d}{d-2}dx\\
&\geq \frac{C_1}{2}.
\end{align*}
Thus, we obtain from $(\ref{l91})$ that for all $s\in\mathbb{R}$
\begin{align}
\label{l100}
\bigg|\bigg\{ t\in [s,s+1]:\int_{|x-x(t)|\leq C_2(\eta)} |u(t,x)|^\frac{2d}{d-2}dx\geq \frac{C_1}{2}\bigg\}\bigg|\geq C_1.
\end{align}
from which we settle the second step.

To conclude the proof, we use ($\ref{l100}$) to obtain the desired estimate ($\ref{l90}$).  Arguing similarly as in ($\ref{l99}$), we obtain 
\begin{align*}
\lVert u(t)\rVert_{L_x^\frac{2d}{d-2}(|x-x(t)|\leq C_2(\eta))}&\leq \lVert u(t)\rVert_{L_x^4(|x-x(t)|\leq C_2(\eta))}^{\theta}\lVert u(t)\rVert_{L_x^\frac{2d}{d-2(1-\epsilon)}(|x-x(t)|\leq C_2(\eta))}^{1-\theta}\\
&\leq \lVert u(t)\rVert_{L_x^4(|x-x(t)|\leq C_2(\eta))}^{\theta}\lVert u(t)\rVert_{L_x^\frac{2d}{d-2(1-\epsilon)}(\mathbb{R}^d)}^{1-\theta}\\
&\leq C\lVert u(t)\rVert_{L_x^4(|x-x(t)|\leq C_2(\eta))}^{\theta}\lVert u(t)\rVert_{\dot{H}_x^{1-\epsilon}(\mathbb{R}^d)}^{1-\theta}\\
&\leq C\lVert u(t)\rVert_{L_x^4(|x-x(t)|\leq C_2(\eta))}^\theta
\end{align*}
for some $\theta\in (0,1)$.

\noindent Then for all $s\in\mathbb{R}$ we have
\begin{align*}
\int_{s}^{s+1} \int_{|x-x(t)|\leq C_2(\eta)} |u(t,x)|^4dxdt&=\int_{s}^{s+1} \lVert u(t)\rVert_{L_x^4(|x-x(t)|\leq C_2(\eta)}^4dt\\
&\geq \int_{s}^{s+1}C^{-4/\theta}\lVert u(t)\rVert_{L_x^\frac{2d}{d-2}(|x-x(t)|\leq C_2(\eta))}^{4/\theta}dt\\
&\geq C_1\cdot C^{-4/\theta}\left(\frac{C_1}{2}\right)^\frac{4(d-2)}{2d\theta}
\end{align*}
where we used $(\ref{l100})$ to obtain the last inequality.  Since $C_1$, $C_2$ and $C$ are independent of $s$, this yields the desired estimate ($\ref{l90}$).
\end{proof}

Having shown the two steps we outlined above, we are now ready to address the proof of the main proposition of this section, which precludes the soliton-like scenario.
\begin{proposition}
Assume $d\geq 6$.  Then there is no $u:\mathbb{R}\times \mathbb{R}^d\rightarrow\mathbb{R}$ such that $u$ solves (NLW) and satisfies the properties of a soliton-like solution in the sense of Theorem $\ref{l15}$.
\end{proposition}

\begin{proof}
We argue as in \cite{KillipVisanSupercriticalNLW3D}.  Suppose for a contradiction that such a solution $u$ existed.  Fix $T>0$ and choose $C$ as in Lemma $\ref{l85}$ and $R$, $c$ as in Lemma $\ref{l89}$.  We then write,
\begin{align}
\int_0^T\int_{\mathbb{R}^d} \frac{|u(t,x)|^4}{|x|}dxdt&\geq \sum_{i=0}^{\lfloor T\rfloor -1} \int_i^{i+1}\int_{|x-x(t)|\leq R} \frac{|u(t,x)|^4}{|x|}dxdt.\label{l101}
\end{align}
Note that for all $i\in \{0,\cdots,\lfloor T\rfloor -1\}$ the conditions $t\in [i,i+1)$ and $x\in \{x\in\mathbb{R}^d:|x-x(t)|\leq R\}$ yield
\begin{align*}
|x|&\leq |x-x(t)|+|x(t)-x(0)|+|x(0)|\leq R+C+t+|x(0)|\leq C'+i.
\end{align*}
Using this bound,
\begin{align}
\nonumber (\ref{l101})&\geq \sum_{i=0}^{\lfloor T\rfloor -1}\frac{1}{C'+i}\int_i^{i+1}\int_{|x-x(t)|\leq R} |u(t,x)|^4dxdt\\
\nonumber &\geq c\sum_{i=0}^{\lfloor T\rfloor-1} \frac{1}{C'+i}\\
&\geq c\int_{0}^{\lfloor T\rfloor}\frac{1}{C'+t}dt.\label{l102}
\end{align}

Combining ($\ref{l101}$) with ($\ref{l102}$) and invoking Theorem $\ref{l7}$, we obtain 
\begin{align*}
c\log \big(\frac{C'+\lfloor T\rfloor}{C'}\big)\leq \int_0^T\int_{\mathbb{R}^d} \frac{|u(t,x)|^4}{|x|}dxdt\leq CE(u_0,u_1).
\end{align*}

Since $u$ is a soliton-like solution, by Theorem $\ref{l56}$ we have $E(u_0,u_1)<\infty$.  Noting that $T>0$ is arbitrary and the constants $C$, $R$ and $c$ are independent of $T$, letting $T$ tend to infinity, we derive a contradiction.  This completes the proof of the proposition.
\end{proof}

\section{Low-to-high frequency cascade solution}

In this section, we rule out the low-to-high frequency cascade scenario identified in Theorem $\ref{l15}$.

\begin{proposition}
There is no $u:\mathbb{R}\times \mathbb{R}^d\rightarrow\mathbb{R}$ such that $u$ solves (NLW), and satisfies the properties of a low-to-high frequency cascade solution in the sense of Theorem $\ref{l15}$.
\end{proposition}

\begin{proof}
We proceed in a similar manner as in \cite{KillipVisanSupercriticalNLS}.  Assume to the contrary that there exists such a solution $u$.  Since $u$ is a low-to-high frequency cascade solution, we may choose a sequence $\{t_n\}\subset\mathbb{R}$ with $t_n\rightarrow\infty$ such that $N(t_n)\rightarrow\infty$ as $n\rightarrow\infty$.

Using ($\ref{l14}$) followed by H\"older's inequality with $u\in L_t^\infty(\dot{H}_x^{1-\epsilon}\times \dot{H}_x^{-\epsilon})$ for some $\epsilon>0$ (Theorem $\ref{l56}$) we have, for all $n\in\mathbb{N}$ and $\eta>0$,
\begin{align}
\nonumber &\int_{|\xi|\leq c(\eta)N(t_n)} |\xi|^{2}|\hat{u}(t_n,\xi)|^2+|\hat{u}_t(t_n,\xi)|^2d\xi\\
\nonumber &\lesssim \left(\int_{|\xi|\leq c(\eta)N(t_n)} |\xi|^{2s_c}|\hat{u}(t_n,\xi)|^2d\xi\right)^{\frac{\epsilon}{\epsilon+s_c-1}}\left(\int_{|\xi|\leq c(\eta)N(t_n)} |\xi|^{2(1-\epsilon)}|\hat{u}(t_n,\xi)|^2d\xi\right)^{\frac{s_c-1}{\epsilon+s_c-1}}\\
\nonumber &+\left(\int_{|\xi|\leq c(\eta)N(t_n)} |\xi|^{2(s_c-1)}|\hat{u}_t(t_n,\xi)|^2d\xi\right)^\frac{\epsilon}{\epsilon+s_c-1}\left(\int_{|\xi|\leq c(\eta)N(t_n)} |\xi|^{-2\epsilon} |\hat{u}_t(t_n,\xi)|^2d\xi\right)^\frac{s_c-1}{\epsilon+s_c-1}\\
\nonumber &\lesssim \eta^\frac{\epsilon}{\epsilon+s_c-1} \lVert (u,u_t)\rVert_{L^\infty(\mathbb{R};\dot{H}_x^{1-\epsilon}\times\dot{H}_x^{-\epsilon})}^\frac{2(s_c-1)}{\epsilon+s_c-1}\\
\label{l103}&\lesssim \eta^\frac{\epsilon}{\epsilon+s_c-1}.
\end{align}

On the other hand, by Chebyshev's inequality
\begin{align}
\nonumber &\int_{|\xi|\geq c(\eta)N(t_n)} |\xi|^2|\hat{u}(t_n,\xi)|^2+|\hat{u}_t(t_n,\xi)|^2d\xi\\
\nonumber &\hspace{0.2in}\leq [c(\eta)N(t)]^{-2(s_c-1)}\int_{\mathbb{R}^d} |\xi|^{2s_c}|\hat{u}(t_n,\xi)|^2+|\xi|^{2(s_c-1)}|\hat{u}_t(t_n,\xi)|^2d\xi\\
\nonumber &\hspace{0.2in}\lesssim [c(\eta)N(t_n)]^{-2(s_c-1)}\lVert (u,u_t)\rVert^2_{L^\infty(\mathbb{R};\dot{H}_x^{s_c}\times \dot{H}_x^{s_c-1})}\\
\label{l104}&\hspace{0.2in}\lesssim [c(\eta)N(t_n)]^{-2(s_c-1)}.
\end{align}
for all $\eta>0$ and $n\in\mathbb{N}$.

To continue, we now estimate the nonlinear term in the energy.  Note that using Sobolev's inequality followed by interpolation with $u\in L_t^\infty\dot{H}_x^{s_c}$, 
\begin{align}
\lVert u(t_n)\rVert_{L_x^4}&\lesssim \lVert |\nabla|^\frac{d}{4}u(t_n)\rVert_{L_x^2}\lesssim \lVert \nabla u(t_n)\rVert_{L_x^2}^{\frac{1}{2}}\lVert u\rVert_{L_t^\infty\dot{H}_x^{s_c}}^{\frac{1}{2}}\lesssim \lVert \nabla u(t_n)\rVert_{L_x^2}^{\frac{1}{2}}.\label{l105}
\end{align}

Combining ($\ref{l103}$), ($\ref{l104}$) and invoking Plancherel's theorem in ($\ref{l105}$),  we estimate the energy as
\begin{align}
\nonumber E(u(t_n),u_t(t_n))&\lesssim \int_{\mathbb{R}^d} |\xi|^2|\hat{u}(t_n)|^2d\xi+\int_{\mathbb{R}^d} |\hat{u}_t(t_n)|^2d\xi+\left(\int_{\mathbb{R}^d} |\xi|^2|\hat{u}(t_n)|^2d\xi\right)^{2}\\
\nonumber &\lesssim \int_{|\xi|\leq c(\eta)N(t_n)} |\xi|^2|\hat{u}(t_n)|^2d\xi+\int_{|\xi|\geq c(\eta)N(t_n)} |\xi|^2\hat{u}(t_n)|^2d\xi\\
\nonumber &+\int_{|\xi|\leq c(\eta)N(t_n)} |\hat{u}_t(t_n)|^2d\xi+\int_{|\xi|\geq c(\eta)N(t_n)} |\hat{u}_t(t_n)|^2d\xi\\
\nonumber &+\left[\int_{|\xi|\leq c(\eta)N(t_n)} |\xi|^2|\hat{u}(t_n)|^2d\xi+\int_{|\xi|\geq c(\eta)N(t_n)} |\xi|^2\hat{u}(t_n)|^2d\xi\right]^{2}\\
&\lesssim \eta^\frac{\epsilon}{\epsilon+s_c-1}+[c(\eta)N(t_n)]^{-2(s_c-1)}+\eta^{\frac{2\epsilon}{\epsilon+s_c-1}}+[c(\eta)N(t_n)]^{-4(s_c-1)},\label{l106}
\end{align}
for all $\eta>0$ and $n\in\mathbb{N}$.

Letting $n\rightarrow\infty$ in ($\ref{l106}$) and using the conservation of energy, now $N(t_n)\rightarrow \infty$ yields for all $\eta>0$,
\begin{align*}
E(u(0),u_t(0))&\lesssim \eta^\frac{\epsilon}{\epsilon+s_c-1}+\eta^{\frac{2\epsilon}{\epsilon+s_c-1}}.
\end{align*}
Taking $\eta\rightarrow 0$, we obtain $E(u(0),u_t(0))=0$.  Thus $u\equiv 0$ contradicting our assumption that $u$ is a blow-up solution.
\end{proof}

\appendix
\section{}
In this appendix, we present the detailed proofs of some observations that we used in the discussion above.  More precisely,
\subsection{The bound ($\ref{l65}$).}
Here, we present the argument used in obtaining the bound ($\ref{l75}$) from the decay estimate ($\ref{l64}$) in the proof of Lemma $\ref{l63}$.  We begin by recalling the following Gronwall inequality from \cite{KillipVisanSupercriticalNLW3D}.
\begin{lemma}
\label{l107} Let $\gamma,\gamma',C,\eta>0$ and $\rho\in (0,\gamma)$ be given such that
\begin{align*}
\eta\leq \frac{1}{4}\min\{1-2^{-\gamma},1-2^{-\gamma'},1-2^{\rho-\gamma}\}.
\end{align*}
Then for every bounded sequence $\{x_k\}\subset\mathbb{R}^+$ satisfying
\begin{align*}
x_k&\leq C2^{-\gamma k}+\eta\sum_{l=0}^{k-1} 2^{-\gamma (k-l)}x_l+\eta\sum_{l=k}^\infty 2^{-\gamma' |k-l|}x_l,
\end{align*}
we have
\begin{align*}
x_k&\leq (4C+\lVert x\rVert_{l^\infty})2^{-\rho k}.
\end{align*}
\end{lemma}

We now turn our attention to the proof of the bound $(\ref{l75}$). 

Fix $\gamma=d-\frac{d}{R}-3$, $\gamma'=\frac{d}{R}-\frac{d}{2}+2$, $C=1$ and $\rho\in (0,\gamma)$.  Let $C'$ be the constant in the inequality given in $(\ref{l64})$ (note that this constant comes from the combinatorial considerations, as well as the constants in each application of the Sobolev and Bernstein inequalities, and thus may be chosen independent of $\eta$ and $N_0$).  

We now choose $\eta>0$ such that 
\begin{align*}
\eta':=C'\eta\leq \left(\frac{1}{4}\min\{1-2^{-\gamma},1-2^{-\gamma'},1-2^{\rho-\gamma}\}\right)^2\end{align*}
and
\begin{align}
\eta'\leq 2^{-4(\gamma+\gamma')}.\label{l108}
\end{align}
Having chosen $\eta$, we may use our hypothesis on $u$ (in the context of the proof of Lemma $\ref{l58}$) to choose $N_0\in\mathbb{N}$ such that
\begin{align*}
\lVert |\nabla|^{s_c}u_{\leq N_0}\rVert_{L^\infty L^2}<\eta.
\end{align*}
For all $k\in\mathbb{N}$, we define $x_k=\mathcal{S}(2^{-k}N_0)$.  Then, applying ($\ref{l64}$) for all $k\geq 0$, we have
\begin{align}
\nonumber x_k&=\mathcal{S}(2^{-k}N_0)\\
\nonumber &\leq C'\left(\frac{2^{-k}N_0}{N_0}\right)^\gamma +C'\eta \sum_{i=0}^{k+2} \left(\frac{2^{-k}N_0}{2^{-i}N_0}\right)^\gamma x_i+C'\eta\sum_{i=k+3}^\infty \left(\frac{2^{-i}N_0}{2^{-k}N_0}\right)^{\gamma'}x_i\\
\nonumber &=C'2^{-k\gamma}+\eta'\sum_{i=0}^{k+2} 2^{(i-k)\gamma}x_i+\eta'\sum_{i=k+3}^\infty 2^{(k-i)\gamma}x_i\\
\nonumber &=C'2^{-k\gamma}+\eta'\sum_{i=0}^{k-1} 2^{-\gamma|k-i|}x_i+\eta'x_k+\eta'2^{[(k+1)-k]\gamma} x_{k+1}\\
\nonumber &\hspace{0.2in}+\eta' 2^{[(k+2)-k]\gamma}x_{k+2}+\eta'\sum_{i=k+3}^\infty 2^{-\gamma'|k-i|}x_i\\
\nonumber &\leq C'2^{-k\gamma}+(\eta')^\frac{1}{2}\sum_{i=0}^k 2^{-\gamma|k-i|}x_i+(\eta')^\frac{1}{2}x_k+(\eta')^\frac{1}{2}2^{-\gamma'} x_{k+1}\\
\nonumber &\hspace{0.2in}+(\eta')^\frac{1}{2}2^{-2\gamma'}x_{k+2}+(\eta')^\frac{1}{2}\sum_{i=k+3}^\infty 2^{-\gamma' |k-i|}x_i\\
&\leq C'2^{-k\gamma}+(\eta')^\frac{1}{2}\sum_{i=0}^{k-1} 2^{-\gamma |k-i|}x_i+(\eta')^\frac{1}{2}\sum_{i=k}^\infty 2^{-\gamma' |k-i|}x_i\label{l109}
\end{align}
where we have used ($\ref{l108}$) and noted that $\eta'<1$ and $2^{-\gamma |k-k|}=2^{-\gamma' |k-k|}=2^0$.  

Applying the estimate ($\ref{l109}$) and invoking Lemma $\ref{l107}$, we obtain the bound
\begin{align*}
x_k\lesssim 2^{-k\rho}.
\end{align*}
Thus, for all $N=2^{-k}N_0\leq 8N_0$, we have
\begin{align*}
\mathcal{S}(N)=\mathcal{S}(2^{-k}N_0)\lesssim (2^{-k})^\rho=N^\rho
\end{align*}
where $\rho\in (0,d-\frac{d}{R}-3)$.  This gives the desired inequality ($\ref{l75}$).

\subsection{Weak continuity of the wave propagator.}
We now recall that the wave propagator $\mathcal{W}(t)$ is weakly continuous for all $t\in\mathbb{R}$, which was used to obtain the inequality ($\ref{l96}$) in the proof of Proposition $\ref{l89}$.

\begin{proposition}
\label{l110}
Suppose $\{(f_n,g_n)\}\subset \dot{H}_x^1\times L_x^2$ is a sequence such that for some $(f,g)\in \dot{H}_x^1\times L_x^2$, we have
\begin{align}
(f_n,g_n)\rightharpoonup (f,g)\quad\textrm{weakly in }\quad \dot{H}_x^1\times L_x^2.
\label{l111}
\end{align}
Then for every $\tau\in \mathbb{R}$, 
\begin{align*}
\mathcal{W}(\tau)(f_n,g_n)\rightharpoonup \mathcal{W}(\tau)(f,g)\quad\textrm{weakly in }\quad L_x^\frac{2d}{d-2}
\end{align*}
\end{proposition}
\begin{proof}
Fix $\tau>0$ and note that by the Strichartz inequality the operators $A:\dot{H}^1_x\rightarrow L_x^\frac{2d}{d-2}$ defined by $Af=\mathcal{W}(\tau)(f,0)$ and $B:L^2_x\rightarrow L_x^\frac{2d}{d-2}$ defined by $Bg=\mathcal{W}(\tau)(0,g)$ are bounded and linear.  Thus, they are weakly continuous and the hypothesis $(\ref{l111})$ implies that  
\begin{align}
\mathcal{W}(\tau)(f_n-f,0)\rightharpoonup 0\quad\textrm{and}\quad\label{l112}\mathcal{W}(\tau)(0,g_n-g)\rightharpoonup 0
\end{align}
weakly in $L_x^\frac{2d}{d-2}$.  

Next, by the linearity of the propagator $\mathcal{W}(\tau)$, we have
\begin{align}
\mathcal{W}(\tau)(f_n,g_n)&=\mathcal{W}(\tau)(f_n-f,0)+\mathcal{W}(\tau)(0,g_n-g)+\mathcal{W}(\tau)(f,g).\label{l113}
\end{align}
Invoking the weak limits ($\ref{l112}$) in ($\ref{l113}$), we obtain the desired weak convergence.
\end{proof}

\section*{Acknowledgements}
The author would like to thank Monica Visan for suggesting the problem and for valuable comments and suggestions on the manuscript, as well as helpful discussions.  The author would also like to express her thanks to her advisors William Beckner and Nata{\u s}a Pavlovi\'c for their generous help, guidance and support.  The author was supported by the John L. \& Anne Crawford Endowed Presidential Fellowship and the Professor \& Mrs. Hubert S. Wall Endowed Presidential Fellowship from the University of Texas at Austin during the preparation of this work.

\end{document}